\documentclass[11pt]{article}
\usepackage[centertags]{amsmath}
\usepackage{amsfonts}
\usepackage{amssymb}
\usepackage{amsthm,color}
\usepackage{newlfont}
\usepackage{bbm}
\usepackage{graphicx}
\usepackage{mathtools}
\usepackage{physics}
\usepackage{enumerate}

\newtheorem{Theorem}{Theorem}[section]

\newtheorem{Proposition}[Theorem]{Proposition}
\newtheorem{Lemma}[Theorem]{Lemma}

\newtheorem{Remark}[Theorem]{Remark}

\newtheorem{Hypothesis}{Hypothesis}

\numberwithin{equation}{section}

\begin{document}

\def\le{\left}
\def\r{\right}
\def\cost{\mbox{const}}
\def\a{\alpha}
\def\d{\delta}
\def\ph{\varphi}
\def\e{\epsilon}
\def\la{\lambda}
\def\si{\sigma}
\def\La{\Lambda}
\def\B{{\cal B}}
\def\A{{\mathcal A}}
\def\L{{\mathcal L}}
\def\O{{\mathcal O}}
\def\bO{\overline{{\mathcal O}}}
\def\F{{\mathcal F}}
\def\K{{\mathcal K}}
\def\H{{\mathcal H}}
\def\D{{\mathcal D}}
\def\C{{\mathcal C}}
\def\M{{\mathcal M}}
\def\N{{\mathcal N}}
\def\G{{\mathcal G}}
\def\T{{\mathcal T}}
\def\R{{\mathbb R}}
\def\I{{\mathcal I}}

\def\bw{\overline{W}}
\def\phin{\|\varphi\|_{0}}
\def\s0t{\sup_{t \in [0,T]}}
\def\lt{\lim_{t\rightarrow 0}}
\def\iot{\int_{0}^{t}}
\def\ioi{\int_0^{+\infty}}
\def\ds{\displaystyle}
\def\pag{\vfill\eject}
\def\fine{\par\vfill\supereject\end}
\def\acapo{\hfill\break}

\def\beq{\begin{equation}}
\def\eeq{\end{equation}}
\def\barr{\begin{array}}
\def\earr{\end{array}}
\def\vs{\vspace{.1mm}   \\}
\def\rd{\reals\,^{d}}
\def\rn{\reals\,^{n}}
\def\rr{\reals\,^{r}}
\def\bD{\overline{{\mathcal D}}}
\newcommand{\dimo}{\hfill \break {\bf Proof - }}
\newcommand{\nat}{\mathbb N}
\newcommand{\E}{\mathbb E}
\newcommand{\Pro}{\mathbb P}
\newcommand{\com}{{\scriptstyle \circ}}
\newcommand{\reals}{\mathbb R}

\def\Amu{{A_\mu}}
\def\Qmu{{Q_\mu}}
\def\Smu{{S_\mu}}
\def\H{{\mathcal{H}}}
\def\Im{{\textnormal{Im }}}
\def\Tr{{\textnormal{Tr}}}
\def\E{{\mathbb{E}}}
\def\P{{\mathbb{P}}}
\def\span{{\textnormal{span}}}

\title{SPDEs on narrow channels and graphs: convergence and large deviations in case of non smooth noise}
\author{Sandra Cerrai\thanks{Partially supported by the NSF grant  DMS-1954299 - {\em Multiscale analysis of infinite-dimensional stochastic systems}}, Wen-Tai Hsu\\
\vspace{.1cm}\\
Department of Mathematics\\
 University of Maryland\\
College Park, Maryland, USA
}

\date{}

\maketitle

\begin{abstract}
We investigate a class of stochastic partial differential equations of reaction-diffusion type defined on graphs, which can be derived as the limit of SPDEs on narrow planar channels. 
In the first part, we demonstrate that this limit can be achieved under less restrictive assumptions on the regularity of the noise, compared to \cite{CF}. 
In the second part, we   establish the validity of a large deviation principle for the SPDEs on the narrow channels and on the graphs, as the width of the narrow channels and the intensity of the noise are jointly vanishing.
\end{abstract}

\section{Introduction}

In the present paper, we deal with the following stochastic reaction-diffusion equations\begin{equation}\label{narrow SPDE}
    \begin{dcases}
        \displaystyle \partial_t u_\epsilon  (t,x,y)
        = \mathcal{L}_\epsilon u_\epsilon (t,x,y)
        + b(u_\epsilon(t,x,y))
        + \partial_t w^Q (t,x,y),
        \ \ \ \  (x,y) \in G, \\[10pt]
        \displaystyle \frac{\partial u_\epsilon }{\partial \nu_\epsilon} (t,x,y) = 0 , \ \ \ \ (x,y) \in 	\partial G, \ \ \ \ \ 
        u_\epsilon(0,x,y)=u_0 (x,y),
    \end{dcases}
\end{equation}
depending on a small parameter $\e>0$, where
\begin{equation*}
    \mathcal{L}_\epsilon \coloneqq \frac{1}{2} \left(\frac{\partial^2 }{\partial x^2} + \frac{1}{\epsilon^2} \frac{\partial^2 }{\partial y^2}\right).
\end{equation*}
Here $G$ is a bounded smooth domain in $\mathbb{R}^2$, $\nu_\e(x,y)$ is the unit interior co-normal at $\partial G$ associated with $\mathcal{L}_\e$, and $w^Q(t)$ is a cylindrical Wiener process in $L^2(G)$.
Under suitable assumptions on the nonlinearity $b$ and the covariance operator $Q$, it can be proved that equation \eqref{narrow SPDE} is well-posed in $L^2(G)$. With the change of variable 
\[(x,y) \in\,G\mapsto (x, \epsilon y) \in\,G_\epsilon:=\{(x,y) \in\,\reals^2\,:\,(x,\epsilon^{-1}y) \in\,G\},\]
equation \eqref{narrow SPDE} can be rewritten as  the following stochastic reaction-diffusion equations on the narrow channel $G_\e$
\begin{equation}\label{narrow SPDE_var}
    \begin{dcases}
        \displaystyle \partial_t v_\epsilon(t,x,y)
        = \frac{1}{2}\, \Delta v_\epsilon (t,x,y)
        + b(v_\epsilon(t,x,y))
        + \sqrt{\epsilon}\,\partial_t w^{Q_\epsilon}(t,x,y),
        \ \  \ \ (x,y) \in G_\epsilon,\\[10pt]
        \displaystyle \frac{\partial v_\epsilon }{\partial \hat{\nu}_\epsilon} (t,x,y) = 0 , \ \ \ \ (x,y) \in 	\partial G_\epsilon, \ \ \  \ \ \ \ \ \ 
        v_\epsilon(0,x,y)=u_0 (x,\epsilon^{-1}y).
    \end{dcases}
\end{equation}
Here  $\hat{\nu}_\e(x,y)$ is the inward unit normal vector at $\partial G_\e$ and $Q_\epsilon:=R_\epsilon\, Q\,R_\epsilon^{-1}$, where the operator $R_\epsilon:L^2(G)\to L^2(G_\epsilon)$ is defined by 
\[(R_\epsilon f)(x,y):=f(x,y/\epsilon),\ \ \ \ \ f \in\,L^2(G),\ \ \ \ (x,y) \in\,G_\e.\]
Equations like \eqref{narrow SPDE} and \eqref{narrow SPDE_var} can serve as models for Brownian motors (ratchets) in statistical physics and molecular dynamics in biology, where the particles or molecules move along a designated track which can be viewed as a tubular domain with several wings. See e.g. \cite{Molecular} and \cite{Brownian}, as well as \cite{fhb} for some details about these systems. In recent years there has been  a substantial interest in the study of deterministic and stochastic PDEs defined on narrow domains and graphs. However, the existing literature about SPDEs on graphs is relatively limited. In addition to \cite{CF}, where the limiting behavior of systems like  \eqref{narrow SPDE} and \eqref{narrow SPDE_var} has been investigated and suitable SPDEs on graphs have been obtained, we recall \cite{CerraiFreidlin2019} and \cite{CerraiXi2021}, where suitable SPDEs on graphs have been obtained from the fast flow asymptotics of stochastic incompressible planar viscous fluids. Additionally,  different viewpoints and approaches to SPDEs on graphs have been explored in works such as \cite{bona} and \cite{Fan2017}.

In the aforementioned paper \cite{CF}, the first named author and M. Freidlin have studied  the limiting behavior of  particles/molecules moving in narrow channels with wings, as the width of the channels and wings vanishes, by working directly with equation \eqref{narrow SPDE}. 
They have shown that $u_\epsilon$ converges  to the solution of an SPDE defined on the graph $\Gamma$, obtained by identifying all points in the same connected component $C_k(x)$ of the cross section
\[C(x):=\{ y \in\,\mathbb{R}\ :\ (x,y) \in\,G\}=:\bigcup_{k=1}^{N(x)} C_k(x),\]
for all $x \in\,\mathbb{R}$ such that there exists $y \in\,\mathbb{R}$ with $(x,y) \in\,G$
(see Figure 1 and  Section \ref{secprelim} for all details). 
More precisely, for every $g:G\to\mathbb{R}$ they have defined the function $g^\wedge:\Gamma\to\mathbb{R}$  by setting
\begin{equation*}
    g^\wedge(x,k):= \frac{1}{l_k(x)} \int_{C_k(x)} g(x,y) dy, \ \ \ \ \ \ \ \ (x,k) \in \Gamma,
\end{equation*} and  have proved that
for every $p \geq 1$ and  $0<\tau<T$\begin{equation}
\label{convintro}
 \lim_{\epsilon \to 0} \mathbb{E} \sup_{t \in\,[\tau,T]} \abs{u_\e(t) - \bar{u}(t)\circ \Pi}_{L^2(G)}^p =    \lim_{\epsilon \to 0} \mathbb{E} \sup_{t \in\,[\tau,T]} \abs{u_\e^\wedge(t) - \bar{u}(t)}_{L^2(\Gamma,\,\nu)}^p = 0,
\end{equation}
where $\Pi$ is the projection of $G$ into $\Gamma$, $\nu$ is a suitable projection on $\Gamma$ of the Lebesgue measure on $G$, that is consistent with the construction of the graph $\Gamma$ (see Subsection 2.3) 
and 
$\bar{u}$ is the solution of the stochastic PDE on the graph $\Gamma$
\begin{equation}
\label{graph SPDE}
        \partial_t\bar{u}(t) = \bar{L}\bar{u}(t)+b(\bar{u}(t)) + \partial_t \bar{w}^Q(t), \ \ \ \ \ \ \ \ \ 
        \bar{u}(0)=u^\wedge_0.
\end{equation}
Here $\bar{w}^Q(t)$ is the projection of the cylindrical Wiener process $w^Q(t)$ on the graph $\Gamma$, and $\bar{L}$ is the generator of a Markov process $\bar{Z}(t)$ on $\Gamma$, obtained in \cite{fw12} as the weak limit in $C([0,T];\Gamma)$ of  $\Pi(Z_\e(t))$, the projection on the graph $\Gamma$ of the  diffusion $Z_\epsilon(t)$ associated with the operator $\mathcal{L}_\epsilon$, endowed with reflecting boundary conditions at $\partial G$.

 The process $Z_\epsilon(t)$ has a fast vertical component, whose invariant measure is given by the normalized Lebesgue measure, and, as a consequence of an averaging principle, Freidlin and Wentcell proved in \cite{fw12} that the process $\Pi(Z_\e(t))$, that describes the slow component of the motion, converges to the Markov process $\bar{Z}(t)$, as $\e\downarrow 0$. Moreover, they gave an explicit description of the  generator $\bar{L}$, in terms of some second order differential operators on each edge of $\Gamma$ and suitable gluing conditions at each internal vertex (see Section \ref{secprelim}).
 Starting from these results, the authors of \cite{CF} have studied the limiting behavior of the semigroup $S_\e(t)$ associated with the diffusion $Z_\epsilon(t)$,  and defined as
 \[S_\epsilon(t)\varphi(z)=\mathbf{E}_z\varphi(Z_\epsilon(t)),\ \ \ \ \ \ z \in\,G,\]
 for every $\varphi:G\to\mathbb{R}$, Borel and bounded. Namely, they have proved that  for every $\varphi \in\,L^2(G)$ and $0<\tau<T$
 \begin{equation}\label{star-40}\lim_{\e\to 0}\vert S_\epsilon(t)\varphi-(\bar{S}(t)\varphi^\wedge)\circ \Pi\vert_{L^2(G)}=0.\end{equation}
Since $u_\epsilon$ satisfies the equation
\begin{equation*}
    u_\epsilon(t) = S_\epsilon(t) u_0 + \int_0^t S_\epsilon(t-s)B(u_\epsilon(s)) ds + w_\e(t),\ \ \ \ \ \ \e>0,
\end{equation*}  
and $\bar{u}(t)$ satisfies the equation
\begin{equation*}
    \bar{u}(t) = \bar{S}(t)u_0^\wedge+\int_0^t\bar{S}(t-s)B(\bar{u}(s))ds+ w_{\bar{L}}(t),
\end{equation*}
where
\[w_\e(t):=\int_0^t S_\epsilon(t-s) dw^Q(s),\ \ \ \ \ \ \ w_{\bar{L}}(t):=\int_0^t \bar{S}(t-s) d\bar{w}^Q(s),\]
limit \eqref{star-40} has allowed them to show that if the noises $w^Q(t)$ and $\bar{w}^Q(t)$ are regular enough to live in $L^2(G)$ and $L^2(\Gamma, \nu)$, respectively, then \eqref{convintro} holds. 

Notice that the assumption that both $w^Q(t)$ and $\bar{w}^Q(t)$ live in functional spaces was fundamental in \cite{CF}. Firstly,  the well-posedness of equation \eqref{graph SPDE} required that the noise $\bar{w}^Q(t)$ was function-valued, as no-smoothing effect of the semigroup $\bar{S}(t)$ was proven. Secondly, and more importantly, limit \eqref{star-40} did not allow to prove the convergence of the stochastic convolution $w_\e(t)$ to $w_{\bar{L}}(t)$.
In the first part of the present paper, we are addressing these problems. In particular,  we  prove that,  under the minimal assumptions on the noise $w^Q(t)$ required for the well-posedness of equation \eqref{narrow SPDE} in the space $L^2(G)$, suitable kernel estimates for the semigroup $S_\e(t)$ transfer to analogous kernel estimates for $\bar{S}(t)$, allowing us to have the well-posedness of equation \eqref{graph SPDE}. Moreover, in view of these estimates, we show  that limit \eqref{convintro} is still valid.

In the second part of the paper, we study the validity of a large deviation principle for equation \eqref{narrow SPDE} and we show how  approximation \eqref{convintro} is stable with respect to the small-noise limit. Actually, we fix an arbitrary $\delta>0$ and we consider the equation in $G$
\begin{equation}\label{narrow SPDE-delta}
    \begin{dcases}
        \displaystyle \partial_t u_\epsilon  (t,x,y)
        = \mathcal{L}_\epsilon u_\epsilon (t,x,y)
        + b(u_\epsilon(t,x,y))
        + \e^{\,\delta/2}\partial_t w^Q (t,x,y),
        \ \ \ \  (x,y) \in G, \\[10pt]
        \displaystyle \frac{\partial u_\epsilon }{\partial \nu_\epsilon} (t,x,y) = 0 , \ \ \ \ (x,y) \in 	\partial G, \ \ \ \ \ 
        u_\epsilon(0,x,y)=u_0 (x,y),
    \end{dcases}
\end{equation}
or, equivalently, the equation  on the narrow channel $G_\epsilon$
\begin{equation}\label{narrow SPDE_var-delta}
    \begin{dcases}
        \displaystyle \partial_t v_\epsilon(t,x,y)
        = \frac{1}{2}\, \Delta v_\epsilon (t,x,y)
        + b(v_\epsilon(t,x,y))
        + \e^{\,(1+\delta)/2}\,\partial_t w^{Q_\epsilon}(t,x,y),
        \ \  \ \ (x,y) \in G_\epsilon,\\[10pt]
        \displaystyle \frac{\partial v_\epsilon }{\partial \hat{\nu}_\epsilon} (t,x,y) = 0 , \ \ \ \ (x,y) \in 	\partial G_\epsilon, \ \ \  \ \ \ \ \ \ 
        v_\epsilon(0,x,y)=u_0 (x,\epsilon^{-1}y).
    \end{dcases}
\end{equation}
Our aim is proving that for every $ 0 < \tau < T$, the family $\{\mathcal{L}(  u_\epsilon )\}_{\e \in\,(0,1)}$ satisfies a large deviation principle in $C([\tau,T];L^2(G))$, with speed $\e^\delta$ and action functional 
    \begin{equation}\label{star-41}
       I_T(f)=\frac{1}{2}\, \inf  \abs{\varphi}_{L^2(0,T;L^2(G))}^2,
    \end{equation}
    where the infimum is taken over all $\varphi \in L^2(0,T;L^2(G))$ such that $f=(\bar{u}^\varphi)|_{[\tau,T]}\circ \Pi \in \,C([\tau,T];L^2(G))$ and $\bar{u}^\varphi$ is the solution to the controlled problem on the graph $\Gamma$
    \begin{equation*}
    \begin{dcases}
        \displaystyle \partial_t\bar{u}^{\varphi}(t,x,k) = \bar{L}\bar{u}^{\varphi}(t,x,k)+b(\bar{u}^{\varphi}(t,x,k))+ (Q\varphi)^\wedge (t,x,k),\\[10pt]
        \displaystyle \bar{u}^{\varphi}(0,x,k)=u_0^\wedge(x,k).
    \end{dcases}
\end{equation*}
In the same way, we obtain that the family $\{\mathcal{L}(  u_\epsilon^\wedge )\}_{\e \in\,(0,1)}$ satisfies a large deviation principle in $C([\tau,T];L^2(\Gamma,\,\nu))$ with speed $\e^\delta$ and action functional $\bar{I}_T$, where  $\bar{I}_T$ is defined as in \eqref{star-41}, with the infimum taken now over all  $\varphi \in L^2(0,T;L^2(G))$ such that $f=(\bar{u}^\varphi)|_{[\tau,T]} \in \,C([\tau,T];L^2(\Gamma,\,\nu))$.  
In particular, this means that limit \eqref{convintro} is consistent with respect the small-noise limit, in the sense that the large deviation  principle for the SPDE on the domain $G$ and, equivalently, on the narrow channel $G_\e$, is described by the same large deviation principle that one has for the SPDEs with small noise on the graph.

\medskip

\medskip

The paper is organized as follows. In Section \ref{secprelim}, we begin by reviewing the assumptions on the domain $G$ and we construct the graph $\Gamma$. We then provide a brief summary of existing results concerning the convergence of $\Pi(Z_\epsilon)$ to $\bar{Z}$,  the convergence of the semigroup $S_\e(t)$ to the semigroup $\bar{S}(t)$ and the convergence of the solution  of equation \eqref{narrow SPDE} to the solution of equation \eqref{graph SPDE}.
Section \ref{secQ} is devoted to the proof of the convergence of the solution of \eqref{narrow SPDE} to the solution of equation \eqref{graph SPDE} under weaker regularity conditions for the noise. In Section \ref{secLDP}, we state the large deviation principle result and give an overview of the weak convergence method and introduce the skeleton equation within our specific framework.
Finally, in Section  \ref{secpriori} and  Section \ref{secpf} we provide the proof of the large deviation principle.


\section{Notations and preliminaries}
\label{secprelim}

In this section, we present a concise overview of the notations we are going to use throughout the paper and recall some preliminary results from \cite{CF} and \cite{fw12}.

\subsection{The domain $G$, the narrow channel $G_\e$ and the graph $\Gamma$}

Let $G$ be a bounded open domain in $\reals^2$, having a  smooth boundary $\partial G$ and satisfying the uniform exterior sphere condition. For any $(x,y) \in\,\partial G$,  we denote by $\nu(x,y)$  the unit inward normal vector at the point $(x,y)$.

Now, for every $\epsilon > 0$, we introduce the narrow channel $G_\epsilon$ associated with $G$
\begin{equation*}
    G_\epsilon \coloneqq \{ (x,y) \in \mathbb{R}^2 :(x,\epsilon^{-1}y) \in G \},
\end{equation*}
and we denote by $\nu^\epsilon (x,y)$ the unit inward normal vector at the point $(x, y) \in \partial G_\epsilon$. Notice that
\begin{equation*}
    \nu^\epsilon (x, \epsilon y) = c_\epsilon(x,y) (\epsilon \nu_1 (x,y) , \nu_2 (x,y) ), \ \ \ (x,y) \in \partial G,
\end{equation*}
for some function $c_\epsilon : \partial G \to [1,+\infty)$ such that 
\begin{equation*}
    \sup_{ \substack{x, y \in \partial G \\  \epsilon > 0 }} c_\epsilon (x,y) = c < +\infty.
\end{equation*}

In what follows, we assume that the region $G$ satisfies the following properties
\begin{enumerate}
\item[-] There exist only finitely many $x \in \mathbb {R}$ such that $(x,y) \in \partial G$ and $\nu_2(x,y) = 0$. 
\item[-] For every $x \in \mathbb {R}$, the cross-section $C(x) = \{ y \in \mathbb{R}: (x,y) \in G \}$ consists of only a finite union of intervals. 
Namely, when $C(x) \neq \emptyset$, there exist $N(x) \in \mathbb{N}$ and intervals $C_1(x) , \cdots , C_{N(x)}(x)$ such that
\begin{equation*}
    C(x) = \bigcup_{k=1}^{N(x)} C_k(x).
\end{equation*}
\item[-] If $x \in \mathbb {R}$ is such that $\nu_2(x,y) \neq 0$, then for any $k=1,\cdots, N(x)$ we have
\begin{equation*}
    l_k(x) = \abs{C_k(x)} >0 .
\end{equation*}
\end{enumerate}

By identifying the points within each connected component $C_k(x)$ of each cross section $C(x)$, we can construct a graph $\Gamma$. This graph consists of a finite number of vertices $O_i$, corresponding to the connected components containing points $(x, y) \in \partial G$ such that $\nu_2(x,y) = 0$, and with a finite number of edges $I_k$, connecting the vertices. 
On our graph there are two different types of vertices, exterior ones, that are connected to only one edge of the graph, and interior ones, that are connected to two or more edges. See Figure 1.

On the graph $\Gamma$, we can introduce a notion of distance as follows. If $z_1 = (x_1, k )$ and $z_2 = (x_2, k)$ belong to the same edge $I_k$ , then their distance is given by $d(z_1,z_2) = \abs{x_1-x_2}$. In the case when $z_1$ and $z_2$ belong to different edges, the distance is determined as
\begin{equation*}
    d(z_1,z_2) = \min \{  d(z_1,O_{i_1})+d(O_{i_1},O_{i_2}) + \cdots + d(O_{i_n},z_2)  \},
\end{equation*}
where the minimum is taken over all possible paths from $z_1$ to $z_2$, through every possible sequence of vertices $O_{i_1}, \cdots, O_{i_n}$, that connects $z_1$ to $z_2$.

Now, any point $z$ on the graph $\Gamma$ can be uniquely characterized by two coordinates: the horizontal coordinate $x$ and the integer $k$ which identifies the edge $I_k$ to which $z$ belongs. It is worth noting that if $z$ is an interior vertex $O_i$, the second coordinate may not have a unique choice since there are two or more edges having $O_i$ as their endpoint.

In the upcoming discussion, we define the mapping $\Pi: G \to \Gamma$ as the identification map that associates points in the domain $G$ to their respective locations on the graph $\Gamma$. For any vertex $O_i$ on the graph $\Gamma$, we denote by $E_i$ the set $\Pi^{-1}(O_i)$ consisting of points $(x,y) \in \partial G$ such that $\nu_2(x,y) = 0$. The set $E_i$ can be one point, several points or an interval. Henceforth, we assume that $G$
satisfies the following condition: 
\begin{enumerate}
  \item[-] For each vertex $O_i$, either $\nu_1(x,y)>0$, for all $(x,y) \in E_i$, or $\nu_1(x,y)<0$, for all $(x,y) \in E_i$.
\end{enumerate}

\begin{figure}{}
\centering
\includegraphics[scale=1]{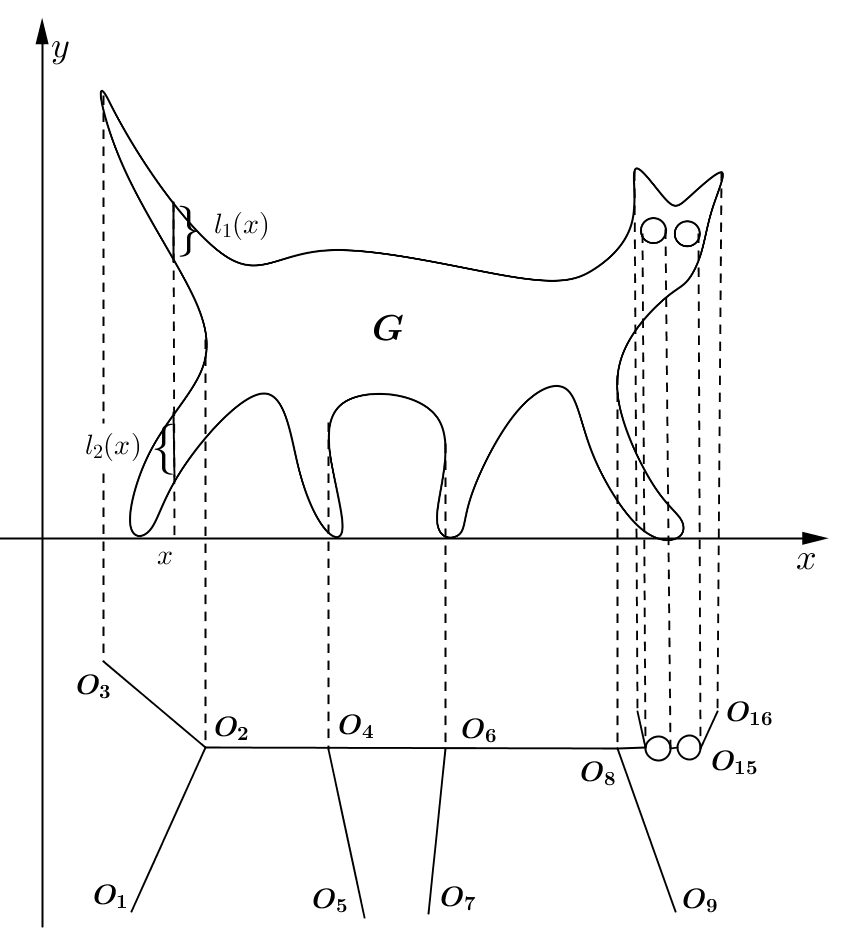}
\caption{The domain $G$ and the graph $\Gamma$}
\end{figure}

\subsection{Limiting results for the transition semigroup}
In order to understand the limiting behavior of the differential operator $\mathcal{L}_\epsilon$, endowed  with the Neumann boundary condition, and its associated semigroup $S_\epsilon(t)$, we consider the following stochastic differential equation with reflecting boundary conditions on the domain $G$
\begin{equation}
    \begin{dcases}
        \displaystyle dX^\epsilon=dB_1(t)
        +\nu_1(X^\epsilon(t), Y^\epsilon(t))d\phi^\epsilon(t), \ \ \ \ X^\epsilon(0) = x, \\[10pt]
        \displaystyle dY^\epsilon=\frac{1}{\epsilon}\,dB_2(t)
        +\frac{1}{\epsilon^2}\,\nu_2(X^\epsilon(t), Y^\epsilon(t))\,d\phi^\epsilon(t), \ \ \ \ Y^\epsilon(0) = y,
    \end{dcases}
\end{equation}
where $(x,y) \in G$.
Such an equation can be rewritten as
\begin{equation*}
    dZ^\epsilon(t) = \sqrt{\sigma_\epsilon}\, dB(t) + \sigma_\epsilon\, \nu (Z^\epsilon(t))\, d\phi^\epsilon(t), \ \ \ \ Z^\epsilon(0) = z,
\end{equation*}
where $z=(x,y) \in G$ and
\begin{equation*}
    \sigma_\epsilon = 
    \begin{pmatrix}
        1 & 0 \\
        0 & \epsilon^{-2}
    \end{pmatrix}.
\end{equation*}
Here $B(t)=(B_1(t),B_2(t))$ is a $2$-dimensional Brownian motion defined on a stochastic basis $(\mathbf{\Omega}, \mathbf{\mathcal{F}}, \{\mathbf{\mathcal{F}}\}_{t\geq 0}, \mathbf{P})$ and $\phi^\epsilon(t)$ is the local time of the process $Z^\epsilon(t)$ on the boundary $\partial G$.  That is, $\phi^\epsilon(t)$ is an  adapted process, continuous with probability 1, non-decreasing and increasing only when $Z^\epsilon(t) \in \partial G$. 
Formally, one can write 
\begin{equation*}
    \phi^\epsilon(t) = \int_0^t  I_{\{ Z^\epsilon(s) \in \partial G \} } d\phi^\epsilon(s).
\end{equation*}
One can show that the generator of $Z^\epsilon$ is $\mathcal{L}_\epsilon$, endowed  with Neumann boundary conditions. 

In \cite{fw12},  a continuous Markov process $\bar{Z}(t)$ on $\Gamma$ is introduced. Its generator  $\bar{L}$ is given by suitable differential operators $\bar{\mathcal{L}}_k$ within each edge $I_k = \{ (x,k) : a_k \leq x \leq b_k \}$ of the graph and is subject to certain gluing conditions at the vertices $O_i$. More precisely, for each $k$ the differential operator $\bar{\mathcal{L}}_k$ has the form 
\begin{equation}
    \bar{\mathcal{L}}_k f(x) = \frac{1}{2l_k(x)} \frac{d}{dx} \left( l_k \frac{df}{dx}\right)(x), \ \ \ a_k < x < b_k,
\end{equation}
and the operator $\bar{L}$, acting on functions $f$ defined on the graph $\Gamma$, is defined as
\begin{equation*}
    \bar{L}f(x,k) 
        =\bar{\mathcal{L}}_k f(x), \ \ \ \text{if $(x,k)$ is an interior point of $I_k$}.
\end{equation*}
The domain $D(\bar{L})$ is defined as the set of continuous functions on the graph $\Gamma$, that are twice continuously differentiable in the interior part of each edge, such that for any vertex $O_i=(x_i,k_{i,1})=\cdots=(x_i,k_{i,n_i})$ there exists finite
\begin{equation*}
    \lim_{(x,k_{i,j}) \to O_i} \bar{L}f(x,k_{i,j}).
\end{equation*}
Moreover, the following one-sided limits exist 
\begin{equation*}
    \lim_{x \to x_i}  l_{k_{i,j}}(x) \frac{df}{dx}(x,k_{i,j}) 
\end{equation*}
along any edge $I_{k_{i,j}}$ ending at the vertex $O_i = (x_i,k_{i,j})$
and the following gluing condition at each vertex $O_i$ is satisfied
\begin{equation}\label{gluing}
    \sum_{k:I_k \sim O_i} \lim_{x \to x_i} \left( \pm l_k(x) \frac{df}{dx}(x,k) \right) = 0,
\end{equation}
where the sign $+$ is taken for right limits and the sign $-$ for left limits.
In the case of an exterior vertex $O_i$, the gluing condition \eqref{gluing} reduces to
\begin{equation}
    \lim_{x \to x_i}  l_k(x) \frac{df}{dx}(x,k)  = 0,
\end{equation}
along the only edge $I_k$ terminating at $O_i$.

In what follows, the transition semigroup associated with the process $\bar{Z}(t)$ is denoted by   $\bar{S}(t)$. In other words, for any Borel and bounded function $f: \Gamma \to \mathbb{R}$, we have
\begin{equation*}
    \bar{S}(t)f(x,k) = \bar{\mathbf{E}}_{(x,k)}f(\bar{Z}(t)),
\end{equation*}
where $\bar{\mathbf{E}}_{x,k}$ is the expectation associated with the Markov process $\bar{Z}(t)$ starting at $(x,k) \in \Gamma$.

In \cite{fw12}   the convergence in distribution of the non-Markov process $\Pi(Z_\epsilon(t))$, which characterizes the slow motion of the process, to the Markov process $\bar{Z}(t)$ in the space $C([0,T];\Gamma)$ is proven. 
\begin{Theorem}\cite[Theorem 1.2]{fw12}
    For any bounded and continuous functional $F$ on $C([0,T];\Gamma)$ and $z \in G$ it holds
    \begin{equation*}
    \lim_{\epsilon \to 0} \mathbf{E}_z F(\Pi(Z^\epsilon(\cdot))) = \bar{\mathbf{E}}_{\Pi(z)} F(\bar{Z}(\cdot)).
	\end{equation*}
\end{Theorem}

\subsection{Functions and operators on the graph $\Gamma$}
In this subsection, we recall the basic definitions of functions and operators defined  on $G$ and $\Gamma$, as introduced in  \cite{CF}, as well as some convergences results which we will use in what follows.

We denote by $H$ the Hilbert space $L^2(G)$, endowed with the scalar product  $\langle \cdot , \cdot \rangle_H$ and the  corresponding norm $\vert\cdot\vert_H$.
Moreover, we denote by $H^1$ the classical Sobolev space $W^{1,2}(G)$, endowed   with the norm
\begin{equation*}
    \abs{f}_{H^1} \coloneqq \abs{f}_H+\abs{\nabla f}_H.
\end{equation*}

Now we denote by $\bar{H}$  the space of measurable functions $f:\Gamma \to \mathbb{R}$ such that
\begin{equation*}
    \sum_{k=1}^N \int_{I_k} \abs{f(x,k)}^2 l_k(x) dx < +\infty
\end{equation*}
(here $N$ is the total number of edges of $\Gamma$).
The space $\bar{H}$ is a Hilbert space once  endowed with the inner product
\begin{equation*}
    \langle f,g \rangle_{\bar{H}} = \sum_{k=1}^N \int_{I_k} f(x,k)g(x,k)l_k(x) dx.
\end{equation*}
Note that if we introduce in $\Gamma$  the measure $\nu$ 
\begin{equation}
\label{nu}
    \nu(A):=\sum_{k=1}^N \int_{I_k} I_A(x,k)l_k(x) dx, \ \ \ A \in \mathcal{B}(\Gamma),
\end{equation}
we have that $\bar{H}=L^2(\Gamma,\nu)$.

Now, for any $u \in H$ we define
\begin{equation*}u^\wedge(x,k):=\frac{1}{l_k(x)}\int_{C_k(x)} u(x,y)dy, \ \ \ \ \ \ \ \  (x,k) \in \Gamma,
\end{equation*}
and for any $f\in \bar{H}$ we define
\begin{equation*}
    f^{\vee}(x,y):=f(\Pi(x,y)), \ \ \ \ \ \ \ \ (x,y)\in G.
\end{equation*}
Moreover, for any $Q \in \mathcal{L}(H)$ and $f \in \bar{H}$, we define
\begin{equation*}
    Q^\wedge f := \left(Q f ^\vee \right)^\wedge,
\end{equation*}
and, similarly, for any $A \in \mathcal{L}(\bar{H})$ and $u \in H$, we define
\begin{equation*}
    A^\vee u := \left(A u ^\wedge \right)^\vee.
\end{equation*}
Below, we summarize some properties which can be derived from the aforementioned definitions.
\begin{Proposition}\cite{CF}
\label{opprop}
For every $u \in H$ and $f \in \bar{H}$ we have
\begin{equation}   \label{star-7}\langle u^\wedge, f \rangle_{\bar{H}} 
            = \langle u, f^\vee \rangle_H,\ \ \ \ \ 
(f^{\vee})^{\wedge} = f.\end{equation}
In particular, if $\{ f_n \}_{n \in \mathbb{N}}$ is an orthonormal basis in $\bar{H}$, then $\{ f_n^\vee \}_{n \in \mathbb{N}}$ is an orthonormal basis in $H$. Moreover, 
\[\Vert f^{\vee}\Vert_ H=\Vert f\Vert_ {\bar{H}},\ \ \ \ \ \ \ 
    \Vert u^{\wedge}\Vert_ {\bar{H}} \leq \Vert u\Vert_ H.\]
Next, for every $Q \in \mathcal{L}(H)$ and $A \in \mathcal{L}(\bar{H})$ we have
\[\lVert Q^\wedge \rVert_{\mathcal{L}(\bar{H})} \leq \lVert Q \rVert_{\mathcal{L}(H)},\]
and
\begin{equation}  \label{star-6}  \lVert A^\vee \rVert_{\mathcal{L}(H)} \leq \lVert A \rVert_{\mathcal{L}(\bar{H})},\ \ \ \ \ \ \ \ 
   \left( A^\vee \right)^\wedge = A.\end{equation}
Finally,  if $Q \in \mathcal{L}_2(H)$, we have 
\[\lVert Q^\wedge \rVert_{\mathcal{L}_2(\bar{H})} \leq \lVert Q \rVert_{\mathcal{L}_2(H)}.\]
    \end{Proposition}

Once introduced all these notations, we can state the convergence result of the semigroup $S_\epsilon(t)$ to the semigroup $\bar{S}(t)$.\begin{Theorem}\cite[Corollary 5.3 and Theorem 5.4]{CF}
\label{prevconv}
	For any $0 <\tau <T$, $\varphi \in C(\bar{G})$ and $z \in G$, we have
    \begin{equation*}
        \lim_{\epsilon \to 0}\, \sup_{t \in [\tau,T]} \abs{S_\epsilon(t)\varphi(z)-\bar{S}(t)^{\vee}\varphi(z)}=0.
    \end{equation*}
 For every $\varphi \in H$ it holds
    \begin{equation*}
        \lim_{\epsilon \to 0}\, \sup_{t \in [\tau,T]} \abs{S_\epsilon(t)\varphi-\bar{S}(t)^{\vee}\varphi}_H
        =\lim_{\epsilon \to 0}\, \sup_{t \in [\tau,T]} \abs{(S_\epsilon(t)\varphi)^{\wedge}-\bar{S}(t)\varphi^{\wedge}}_{\bar{H}}=0.
    \end{equation*}
    Moreover, $\bar{S}(t)$ extends to a contraction semigroup on $\bar{H}$.
\end{Theorem}

\subsection{Limiting results for the SPDE on $G$}
We recall that $G\subset \mathbb{R}^2$ is a bounded open set with regular boundary. In what follows, we shall assume the following conditions.
\begin{Hypothesis}
\label{hypB}
    The nonlinearity $b: \mathbb{R} \to \mathbb{R}$ is Lipschitz continuous.
\end{Hypothesis}
In particular, this means that
\begin{equation*}
    B: H \to H, \ \ \ u \in H \mapsto B(u) := b \circ u \in H
\end{equation*}
is well-defined and Lipschitz continuous. Similarly, we have that $B: \bar{H} \to \bar{H}$ is well defined and Lipschitz continuous. Moreover, it is immediate to check that for every $f \in\,\bar{H}$
\begin{equation}
\label{star-30}
B(f^\vee)=	(B(f))^\vee.
\end{equation}

We fix a bounded linear operator $Q$ in $H$ such that
\begin{equation}
\label{onbQ}
    Q e_k = \lambda_k e_k
\end{equation}
for some orthonormal basis $\{ e_k \}_{k \in \mathbb{N}}$ in $H$ and some sequence $\{ \lambda_k \}_{k \in \mathbb{N}}\subset [0,\infty)$,
and we assume that $w^Q(t)$ is a cylindrical Wiener process with covariance $Q^2$. This means that it can be written as 
\begin{equation*}
    w^Q(t) = \sum_{k=1}^\infty \lambda_k e_k\beta_k(t),
\end{equation*}
where $\{\beta_k(t)\}_{k \in \mathbb{N}}$ is a sequence of mutually independent Brownian motions defined on some stochastic basis $(\Omega, \mathcal{F}, \{ \mathcal{F}_t\}, \mathbb{P}, \mathbb{E})$.

In \cite{CF}, it is assumed that the series above converges in $L^2(\Omega;H)$. Namely, the following  stronger condition on the operator $Q$ is assumed
  \begin{equation}
  \label{hyp2}	
\lVert Q \rVert_{\mathcal{L}_2(H)}^2
        = \sum_{k=1}^\infty \abs{Qe_k}_H^2 
        = \sum_{k=1}^\infty \lambda_k^2 < + \infty.
    \end{equation}

Under these assumptions (but in fact, weaker conditions on $w^Q$ can be assumed, see section \ref{secQ} for more detail), for any $\epsilon>0$ equation \eqref{narrow SPDE} admits a unique mild solution $u_\epsilon$ in $L^p(\Omega, C([0,T];H))$ (see e.g. \cite{DPZ}). More precisely, for any $T>0$ and $p \geq 1$, there exists a unique adapted process $u_\epsilon \in L^p(\Omega, C([0,T];H))$  such that
\begin{equation*}
    u_\epsilon(t) = S_\epsilon(t) u_0 + \int_0^t S_\epsilon(t-s)B(u_\epsilon(s)) ds + \int_0^t S_\epsilon(t-s) dw^Q(s).
\end{equation*}  
In addition,  the process
\begin{equation*}
    \bar{w}^Q(t) \coloneqq w^Q(t)^\wedge =  \sum_{k=1}^\infty (Qe_k)^\wedge\beta_k(t)=\sum_{k=1}^\infty \lambda_k\, e_k^\wedge\,\beta_k(t),\ \ \ \ t\geq 0,
\end{equation*}
is well-defined in $L^2(\Omega,\bar{H})$. This, together with the fact that $\bar{S}(t)$ is a contraction semigroup on $\bar{H}$, implies that equation \eqref{graph SPDE} is well-posed in $L^p(\Omega, C([0,T];\bar{H}))$ and the solution $\bar{u}(t)$ satisfies the equation
\begin{equation*}
    \bar{u}(t) = \bar{S}(t)u_0^\wedge+\int_0^t\bar{S}(t-s)B(\bar{u}(s))ds+ \int_0^t \bar{S}(t-s) d\bar{w}^Q(s).
\end{equation*}

As we mentioned in the Introduction, in \cite{CF} it has been proved that the solution to \eqref{narrow SPDE} converges to the solution to \eqref{graph SPDE}. More precisely
\begin{Theorem}\label{prevmain}\cite[Theorem 7.2]{CF}
For every  $u_0 \in\,H$, $p\geq 1$ and $T>0$, and for every $\tau \in\,(0,T)$ we have  \begin{equation}
        \lim_{\epsilon \to 0}\, \mathbb{E} \sup_{t \in [\tau, T]} \abs{u_\epsilon(t)-\bar{u}(t)^\vee}_H^p 
        = \lim_{\epsilon \to 0}\, \mathbb{E} \sup_{t \in [\tau, T]} \abs{u_\epsilon(t)^\wedge-\bar{u}(t)}_{\bar{H}}^p =0.
    \end{equation}
\end{Theorem}

\section{Convergence of solutions in case of rough noise}
\label{secQ}

In this section we show that Theorem \ref{prevmain} holds under a  condition on the regularity of the noise that is considerably weaker than \eqref{hyp2}. 
We recall that since $Q \in\,\mathcal{L}^+(H)$ there exists a complete orthonormal system $\{e_k\}_{k \in\,\mathbb{N}}$ in $H$ and a sequence $\{\lambda_k\}_{k \in\,\mathbb{N}} \subset [0,+\infty)$ such that
$Qe_k =\lambda_k e_k$.
In what follows, we shall assume that the eigenvalues $\lambda_k$ satisfy the following condition.
\begin{Hypothesis}\label{finalhypQ}
   There exists $r<\infty$ such that 
\begin{equation*}
      \kappa_{Q, r}:=   \sum_{k=1}^\infty \lambda_k^r  < +\infty.
\end{equation*}
\end{Hypothesis}
Note that, since we are not  assuming that $r=2$, the cylindrical Wiener process $w^Q(t)$ is not well-defined in $H$. In fact, $w^Q(t)$ lives in some larger space Hilbert $U \supset H$ such that the embedding $U_0 \coloneqq Q(H) \xhookrightarrow{} U$ is Hilbert-Schmidt.
Similarly,
\begin{equation*}
    \bar{w}^Q(t) =  \sum_{k=1}^\infty \lambda_k\, e_k^\wedge\,\beta_k(t)
\end{equation*}
is not well-defined in $\bar{H}$ but lives in some larger space $\bar{U} \supset \bar{H}$ with the Hilbert-Schmidt embedding $\bar{U}_0 \coloneqq Q^\wedge (\bar{H}) \xhookrightarrow{} \bar{U}$ (see e.g. \cite{DPZ} for more detail).
Moreover, Hypothesis \ref{finalhypQ} is the necessary  assumption  to ensure the well-posedness of the classical stochastic heat equation in  dimension $d=2$. 

\begin{Theorem}\label{newconv}
    Under Hypotheses \ref{hypB} and \ref{finalhypQ}, for any $u_0 \in H$, $p \geq 1$, and $0 < \tau <T$, it holds

\begin{equation}\label{3.1}
        \lim_{\epsilon \to 0}\, \mathbb{E} \sup_{t \in [\tau, T]} \abs{u_\epsilon(t)-\bar{u}(t)^\vee}_H^p 
        = \lim_{\epsilon \to 0}\, \mathbb{E} \sup_{t \in [\tau, T]} \abs{u_\epsilon(t)^\wedge-\bar{u}(t)}_{\bar{H}}^p =0.
\end{equation}

\end{Theorem}

We have
\[
\begin{aligned}
u_\epsilon(t)&-\bar{u}(t)^\vee=S_\epsilon(t)u_0-\bar{S}(t)^\vee u_0\\[10pt]
&+	\int_0^t 	\left(S_\epsilon(t-s)B(u_\epsilon(s))-\bar{S}(t)^\vee B(\bar{u}(s))\right)\,ds+w_\epsilon(t)-w_{\bar{L}}(t)^\vee,
\end{aligned}
\]
where
\begin{equation}  \label{star-30}w_\epsilon(t):=\int_0^t S_\epsilon(t-s)\,dw^Q(s),\ \ \ \ \ \ \ w_{\bar{L}}(t):=\int_0^t\bar{S}(t-s)\,d\bar{w}^Q(s).\end{equation}
Thanks to Theorem \ref{prevconv}, we can proceed as in  \cite[Proof of Theorem 7.2]{CF} and we have that \eqref{3.1} follows once we show that
\begin{equation}
\label{star-2}
\lim_{\epsilon\to 0}\mathbb{E}\,\sup_{t \in\,[0,T]}\vert w_\epsilon(t)-w_{\bar{L}}(t)^\vee\vert_H^p=0.	
\end{equation}

\subsection{Proof of \eqref{star-2}}
In order to prove \eqref{star-2}, we will need some preliminary results.

\begin{Lemma}\label{Qest1}
    Under Hypothesis \ref{finalhypQ}, for all $T>0$ we have 
        \begin{equation}\label{star-5}
        \sup_{\e \in\,(0,1)}\Vert S_\e(s)\Vert^2_{\mathcal{L}_2(U_0,H)}=
        \sup_{\e \in\,(0,1)}\,\sum_{k=1}^\infty \abs{S_\epsilon(s)Qe_k}_H^2 \leq c_T \,\kappa_{Q, r}^{2/r}  \,s^{-\frac{r-2}{r}},\ \ \ \ s \in\,(0,T].
    \end{equation}
    Moreover, for every $\eta>0$ there exists $N_\eta \in\,\mathbb{N}$ such that
    \begin{equation}
    \label{star-3}
  \sup_{\e \in\,(0,1)}\,  \sum_{k\geq N_\eta}	\abs{S_\epsilon(s)Qe_k}_H^2\leq \eta\,s^{-\frac{r-2}{r}},\ \ \ \ s \in\,(0,T].
    \end{equation}

\end{Lemma}
    \begin{proof}
   By H\"older inequality we have
        \begin{equation}\label{star-4}
            \sum_{k=1}^\infty \abs{S_\epsilon(s)Qe_k}_H^2 
            =\sum_{k=1}^\infty \lambda_k^2 \abs{S_\epsilon(s)e_k}_H^2 
            \leq \kappa_{Q, r}^{2/r}
              \left(\, \sum_{k=1}^\infty \abs{S_\epsilon(s)e_k}_H^{\frac{2r}{r-2}} \right)^{\frac{r-2}{r}} 
            =  \kappa_{Q, r}^{2/r}\,J_\epsilon(s)^{\frac{r-2}{r}},
        \end{equation}
        where
        \begin{equation*}
            J_\epsilon(s) \coloneqq \sum_{k=1}^\infty \abs{S_\epsilon(s)e_k}_H^{\frac{2r}{r-2}}
            = \sum_{k=1}^\infty \abs{S_\epsilon(s)e_k}_H^2 \abs{S_\epsilon(s)e_k}_H^{\frac{4}{r-2}}
            \leq \sum_{k=1}^\infty \abs{S_\epsilon(s)e_k}_H^2
        \end{equation*}
        (last inequality is due to the fact that $S_\epsilon(t)$ is a contraction semigroup on $H$ and $\{ e_k \}$ is an orthonormal basis for $H$). 

        Now, for any $z=(x,y), z'=(x',y') \in G$, we denote by $K_\epsilon(s,z,z')$ the  heat kernel of the semigroup $S_\epsilon(s)$. Namely
        \[S_\epsilon(s)u_0(z)=\int_G K_\epsilon(s,z,z') u_0(z')dz'=\langle K_\epsilon(s,z,\cdot),u_0\rangle_H,\ \ \ \ \ \ \ u_0 \in H,\ \ \ z \in\,G.\]
     According to   \cite[Theorem 3.2.9]{HKST},    the following  heat kernel estimate holds
        \begin{equation}\label{HKE for G}
            K_\epsilon(t,z,z')\leq \frac{c_T}{t}\exp{\left(-\frac{\epsilon^2 \left(\abs{x-x'}^2+\abs{y-y'}^2 \right)}{c\,t}\right)} \leq \frac{c_T}{t},\ \ \ \ \ t \in\,[0,T]
        \end{equation}
        for some $c, c_T>0$.
  This allows to conclude that
        \begin{equation*}
            J_\epsilon(s) \leq \sum_{k=1}^\infty \abs{S_\epsilon(s)e_k}_H^2 
            = \sum_{k=1}^\infty \int_G \vert\langle K_\epsilon(s,z,\cdot),e_k\rangle_H\vert^2\,dz=\int_G\vert K_\epsilon(s,z,\cdot)\vert_H^2\,dz\leq \frac{c_{T,G}}{s}.
                    \end{equation*}
  In particular, due to \eqref{star-4}, this implies \eqref{star-5}.
  
  Now, for every $N \in\,\mathbb{N}$, we have
  \[\begin{aligned}
  	\sum_{k\geq N}& \abs{S_\epsilon(s)Qe_k}_H^2 
            =\sum_{k\geq N} \lambda_k^2 \abs{S_\epsilon(s)e_k}_H^2\\[8pt]
            & 
            \leq \left(\,\sum_{k\geq N}\lambda_k^r\right)^{2/r}
              \left(\, \sum_{k=1}^\infty \abs{S_\epsilon(s)e_k}_H^{\frac{2r}{r-2}} \right)^{\frac{r-2}{r}}=\left(\,\sum_{k\geq N}\lambda_k^r\right)^{2/r}\,J_\epsilon(s)^{\frac{r-2}r},
  \end{aligned}\]
              so that, for every $\epsilon \in\,(0,1)$, we get
\[\sum_{k\geq N} \abs{S_\epsilon(s)Qe_k}_H^2\leq c_{T}\,\left(\,\sum_{k\geq N}\lambda_k^r\right)^{2/r} s^{-\frac{r-2}{2}},\ \ \ \ \ s \in\,(0,T],\]
As a consequence of Hypothesis \ref{finalhypQ}, this implies that for every $\eta>0$ we can find $N_\eta \in\,\mathbb{N}$ such that \eqref{star-3} holds.
  
    \end{proof}
    
    Next, we want to prove an analogous of Lemma \ref{Qest1} for $\bar{S}(t)^\vee$.

\begin{Remark}{\em 
	In order to prove  a result analogous to Lemma \ref{Qest1} for $\bar{S}(t)^\vee$, one should first show that $\bar{S}(t)^\vee$ admits a kernel and an  estimate analogous to \eqref{HKE for G} holds for $\bar{S}(t)^\vee$. There are two possible approaches to obtain this kind of results.	\begin{enumerate}
		\item On the one hand, Theorem  \ref{prevconv} implies that $K_\epsilon(t,z,\cdot)$ converges to some $\hat{K}(t,z,\cdot)$ weakly, as $\e\downarrow 0$, for every $t>0$ and $z \in\,G$, and the limit $\hat{K}(s,z,z')$ is a kernel for $\bar{S}(t)^\vee$. Namely 
		\begin{equation*}
			\bar{S}^\vee(t) u_0(z) = \int_G \hat{K}(t,z,z') u_0(z') dz'=\langle \bar{K}(t,z,\cdot),u_0\rangle_H,\ \ \ \ \ \ u_0 \in\,H.
		\end{equation*}
		As a result, the uniform heat kernel estimate \eqref{HKE for G} remains valid for $\hat{K}(t,z,z')$, that is
		\begin{equation}\label{HKE for graph weak}
			\hat{K}(t,z,z') \leq \frac{c_T}{t},\ \ \ \ \ t \in\,(0,T].
		\end{equation}
		\item On the other hand, one could work directly with  the semigroup $\bar{S}(t)$ on the graph $\Gamma$ and show that it admits a kernel $\bar{K}(t,x,x')$. In this regard, $\bar{K}$ would be a one-dimensional heat kernel and  one can expect that the following bound 
		\begin{equation}\label{HKE for graph strong}
			\bar{K}(s,x,x') \leq \frac{c}{s^{1/2}}\exp\left(-\frac{d(x,x')^2}{c\,s}\right)
		\end{equation}
		should be satisfied. 
		Such a result can be achieved by checking the volume doubling property and the Poincar\'e inequality, which would allow to get the two-sided heat kernel estimate thanks to  the result in \cite{LSC book}.
		Alternatively, one could utilize the result in \cite{Fan2017} by recognizing that the Dirichlet form 
\[\mathcal{E}(f,g) \coloneqq \langle -\bar{L}f , g \rangle_{\bar{H}}\] fits into their setting.
		Another approach could be to use the analysis of the transition probability in \cite{FS2000} to achieve the desired order $s^{-\frac{1}{2}}$.
	\end{enumerate}
	However, for the purpose of this paper  bound \eqref{HKE for graph weak} is sufficient and we will not address bound \eqref{HKE for graph strong}. }
\end{Remark}

\begin{Lemma}\label{Qest2}
    If Hypothesis \ref{finalhypQ} holds for some $r \in (2,\infty)$, then there exists $c_T>0$ such that
    \begin{equation}
\label{star-8}        \Vert \bar{S}^\vee (s)\Vert^2_{\mathcal{L}_2(U_0,H)}=\sum_{k=1}^\infty \abs{\bar{S}^\vee (s)Q e_k}_H^2 \leq c_T\,   \kappa_{Q, r}^{2/r}\, s^{-\frac{r-2}{r}},\ \ \ \ \ \ \ s \in (0,T].
    \end{equation}
    Moreover, for every $\eta>0$ there exists $N_\eta \in\,\mathbb{N}$ such that
    \begin{equation}
    \label{star-3-bis}
 \sum_{k\geq N_\eta}	\abs{\bar{S}(s)^\vee Qe_k}_H^2\leq \eta\,s^{-\frac{r-2}{r}},\ \ \ \ s \in\,(0,T].
    \end{equation}
    \end{Lemma}
\begin{proof}
 The family $\bar{S}(t)^\vee$, $t\geq 0$, defines a semigroup of bounded linear operators in $H$. Actually, due to \eqref{star-7}, for every $t, s\geq 0$ and $g \in\,H$ we have
    \[\begin{aligned}
\bar{S}(t+s)^\vee g&=\left(\bar{S}(t+s) g^\wedge\right)^\vee=\left(\bar{S}(t)\bar{S}(s) g^\wedge\right)^\vee=\left(\bar{S}(t)((\bar{S}(s) g^\wedge)^\vee)^\wedge\right)^\vee\\[10pt]
&=\left(\bar{S}(t)(\bar{S}(s)^\vee g)^\wedge\right)^\vee=\bar{S}(t)^\vee (\bar{S}(s)^\vee g).\end{aligned}
\]
Moreover, since $\bar{S}(t)$ 	is a contraction, due to \eqref{star-6} we have that $\bar{S}(t)^\vee$ is also a contraction.  

Therefore, since $\bar{S}(t)^\vee$ is a semigroup of contractions in $H$ and \eqref{HKE for graph weak} holds, the proof is the same as the one in Lemma \ref{Qest1}.
\end{proof}

As a consequence of Lemma \ref{Qest1} and Lemma \ref{Qest2},  we have the following bounds.

\begin{Lemma}\label{hypAbstractQ}
    Under Hypothesis \ref{finalhypQ}, there exists some $\alpha \in (0,\frac{1}{2})$ such that
    \begin{equation}
        \sup_{\e \in\,(0,1)}\,\int_0^T s^{-2\alpha} \lVert S_\epsilon(s) \rVert_{\mathcal{L}_2(U_0;H)}^2 ds < +\infty, \label{bddeps} 
    \end{equation}
    and
    \begin{equation}
        \int_0^T s^{-2\alpha} \lVert \bar{S}^\vee(s) \rVert_{\mathcal{L}_2(U_0;H)}^2 ds < +\infty. \label{bddbar} 
    \end{equation}
    Moreover
    \begin{equation}
    \label{star-20}
    \lim_{\e\to 0}\int_0^T\lVert S_\epsilon(s)-\bar{S}^\vee(s) \rVert_{\mathcal{L}_2(U_0;H)}^2 ds=0.	
    \end{equation}

\end{Lemma}
\begin{proof}
    In view of Lemma \ref{Qest1}, 
    \begin{equation*}
    \begin{aligned}
        \int_0^T s^{-2\alpha} \lVert S_\epsilon(s)& \rVert_{\mathcal{L}_2(U_0;H)}^2 ds
        =\sum_{k=1}^\infty \int_0^T s^{-2\alpha} \abs{S_\epsilon(s)Qe_k}_H^2 ds\leq   c_T\,\kappa_{Q, r}^{2/r} \int_0^T s^{-2\alpha} s^{-\frac{r-2}{r}} ds.
    \end{aligned}
    \end{equation*}   
    The integral is above is finite if and only if 
    \begin{equation*}  
         2\alpha+\frac{r-2}{r} = 1-2(1/r-\alpha) < 1.
    \end{equation*}
    Thus, if we  pick $\alpha \in \left(0, \frac{1}{r} \right)$ we obtain \eqref{bddeps}. In the same way, due to \eqref{star-8}, we obtain \eqref{bddbar}. 
    
Finally, let us prove \eqref{star-20}. Since both $S_\e(t)$ and $\bar{S}(t)^\vee$ are contractions, for every $N \in\,\mathbb{N}$ we have
\[\begin{aligned}
\lVert S_\epsilon(s)&-\bar{S}^\vee(s) \rVert_{\mathcal{L}_2(U_0;H)}^2=\sum_{k\leq N}\la_k^2\vert 	(S_\epsilon(s)-\bar{S}^\vee(s))e_k\vert_H^2+\sum_{k> N}\la_k^2\vert 	(S_\epsilon(s)-\bar{S}^\vee(s))e_k\vert_H^2\\[8pt]
&\leq \|Q\|^2_{\mathcal{L}(H)}\sum_{k\leq N}\vert 	(S_\epsilon(s)-\bar{S}^\vee(s))e_k\vert_H^2+c\,\left(\,\sum_{k> N}\la_k^r\right)^{\frac 2r}\left(\,\sum_{k> N}\vert (S_\epsilon(s)-\bar{S}^\vee(s))e_k\vert_H^2\right)^{\frac{r}{r-2}}
\end{aligned}\]
By proceeding as in the proofs of Lemma \ref{Qest1} and Lemma \ref{Qest2}, we have
\[\sum_{k> N}\vert (S_\epsilon(s)-\bar{S}^\vee(s))e_k\vert_H^2\leq c_T\,s^{-1}.\]
Moreover, due to Hypothesis \ref{finalhypQ}, for every $\eta>0$ we can find $N_\eta \in\,\mathbb{N}$ such that
 \[\left(\,\sum_{k> N_\eta}\la_k^r\right)^{\frac 2r}\leq \eta.\]
 Thus, we get
 \[\int_0^T \lVert S_\epsilon(s)-\bar{S}^\vee(s) \rVert_{\mathcal{L}_2(U_0;H)}^2\,ds\leq c\sum_{k\leq N_\eta}\int_0^T \vert(S_\epsilon(s)-\bar{S}^\vee(s))e_k\vert_H^2\,ds+c_T\,\eta\int_0^T s^{-\frac{r}{r-2}}\,ds.\]
  According to Theorem \ref{prevconv} and the Dominated Convergence theorem, we have
  \[\lim_{\epsilon\to 0} \sum_{k\leq N_\eta}\int_0^T \vert(S_\epsilon(s)-\bar{S}^\vee(s))e_k\vert_H^2\,ds=0,\]
  and then, due to the arbitrariness of $\eta>0$, we obtain \eqref{star-20}.
    
    \end{proof}

\begin{Remark}
   {\em  As a consequence of  \eqref{bddeps} and \eqref{bddbar}, for every $T>0$, $p \geq 1$ and $\epsilon >0$  we obtain the well-posedness in $L^p(\Omega; C([0,T];H))$ for equation \eqref{narrow SPDE} and the well-posedness in $L^p(\Omega; C(0,T;\bar{H}))$ for equation \eqref{graph SPDE}.}
\end{Remark}

Now, we can prove \eqref{star-2}.
    Let $\alpha \in (0,\frac{1}{2})$ as in Lemma \ref{hypAbstractQ}. From the stochastic factorization formula we get that for any $t \in [0,T]$
    \begin{equation*}
    \begin{aligned}
        \frac{\pi}{\sin \pi \alpha}(w_\epsilon(t)-w_{\bar{L}}(t)^\vee) =\int_0^t &(t-s)^{\alpha-1}S_\epsilon(t-s) Y_{\alpha,1}^\epsilon(s) ds\\[10pt]
        &+\int_0^t (t-s)^{\alpha-1}\left(S_\epsilon(t-s)-\bar{S}(t-s)^\vee \right) Y_{\alpha,2}(s) ds,
    \end{aligned}
    \end{equation*}
    where
    \begin{equation*}
        Y_{\alpha,1}^\epsilon(s) \coloneqq \int_0^s (s-\sigma)^{-\alpha}\left(S_\epsilon(s-\sigma)-\bar{S}(s-\sigma)^\vee \right) dw^Q(\sigma)
    \end{equation*}
    and
    \begin{equation*}
        Y_{\alpha,2}(s) \coloneqq \int_0^s (s-\sigma)^{-\alpha}\bar{S}(s-\sigma)^\vee dw^Q(\sigma).
    \end{equation*}
    Hence, for any $p > \frac{1}{\alpha}$, it holds
    \begin{equation*}
   	\begin{aligned}
        \mathbb{E}\sup_{t \in [0,T]} &\abs{w_\epsilon(t)-w_{\bar{L}}(t)^\vee}_H^p \\[8pt]
        &\leq  c_{p,\alpha,T} \left(\int_0^T \mathbb{E}\abs{Y_{\alpha,1}^\epsilon(s)}_H^p ds 
        +  \mathbb{E} \sup_{t \in\,[0,T]}\int_0^t \abs{ \left (S_\epsilon(t-s)-\bar{S}(t-s)^\vee \right)Y_{\alpha,2}(s)}_H^p ds\right).
    \end{aligned}
    \end{equation*}
    Now, for any fixed $s \geq 0$, we have
    \begin{equation*}
        \mathbb{E}\abs{Y_{\alpha,1}^\epsilon(s)}_H^p
        = c_p \left(\,\sum_{k=1}^\infty \int_0^s \sigma^{-2\alpha} \abs{S_\epsilon(\sigma)Qe_k-\bar{S}(\sigma)^\vee Qe_k}_H^2 d\sigma \right)^{p/2}.
    \end{equation*}
  According to \eqref{star-3} and \eqref{star-3-bis} for every $\eta>0$, we can find $N_{T,\eta} \in \mathbb{N}$ such that
    \begin{equation}\label{star-9}
        \sum_{k>N_{T, \eta}} \int_0^s \sigma^{-2\alpha} \abs{S_\epsilon(\sigma)Qe_k-\bar{S}(\sigma)^\vee Qe_k}_H^2 d\sigma
        \leq \eta.
    \end{equation}
    Moreover, as a consequence of  Theorem \ref{prevconv}, we have
    \begin{equation*}
        \lim_{\epsilon \to 0} \sum_{k\leq N_{T,\eta}} \int_0^s \sigma^{-2\alpha} \abs{S_\epsilon(\sigma)Qe_k-\bar{S}(\sigma)^\vee Qe_k}_H^2 d\sigma = 0.
    \end{equation*}
    Due to the  arbitrariness of $\eta>0$,  this, together with \eqref{star-9}, allows to conclude that 
    \begin{equation*}
        \lim_{\epsilon \to 0} \int_0^T \mathbb{E} \abs{Y_{\alpha,1}^\epsilon(s)}_H^p ds =0.
    \end{equation*}
    
    Next, let us fix $0 < \tau < T$. If $t \in\,[0,\tau]$ we have
    \begin{equation}\label{star-10}
        \int_0^t \abs{(S_\epsilon(t-s)-\bar{S}(t-s)^\vee)Y_{\alpha,2}(s)}_H^p ds \leq c\,\sqrt{\tau}\,\left(\int_0^T \abs{Y_{\alpha,2}(s)}_H^{2p} ds \right)^{1/2}.
        \end{equation}
        On the other hand, if $t \in\,(\tau,T]$, we have
        \begin{equation*}\begin{aligned}
      \int_0^t &\abs{(S_\epsilon(t-s)-\bar{S}(t-s)^\vee)Y_{\alpha,2}(s)}_H^p ds\\[10pt]
      \quad &  =  \int_0^{t-\tau} \abs{(S_\epsilon(t-s)-\bar{S}(t-s)^\vee)Y_{\alpha,2}(s)}_H^p ds
        +  \int_{t-\tau}^t \abs{(S_\epsilon(t-s)-\bar{S}(t-s)^\vee)Y_{\alpha,2}(s)}_H^p ds \\[10pt]
       &\quad  \quad  \leq \int_0^T \sup_{r \in [\tau,T]} \abs{(S_\epsilon(r)-\bar{S}(r)^\vee)Y_{\alpha,2}(s)}_H^p\, ds
        +c\, \sqrt{\tau} \left(\int_0^T \abs{Y_{\alpha,2}(s)}_H^{2p} ds \right)^{1/2}.
    \end{aligned}
    \end{equation*}
    Therefore, if we combine this inequality with \eqref{star-10}, we conclude that     \begin{equation*}\begin{aligned}
      \sup_{t \in\,[0,T]}\,\int_0^t &\abs{(S_\epsilon(t-s)-\bar{S}(t-s)^\vee)Y_{\alpha,2}(s)}_H^p ds\\[10pt]
     &\quad  \leq \int_0^T \sup_{r \in [\tau,T]} \abs{(S_\epsilon(r)-\bar{S}(r)^\vee)Y_{\alpha,2}(s)}_H^p\, ds
        +c\, \sqrt{\tau} \left(\int_0^T \abs{Y_{\alpha,2}(s)}_H^{2p} ds \right)^{1/2}.\end{aligned}
        \end{equation*} 
    Due to \eqref{bddbar}, we have
    \begin{equation*}
  \mathbb{E}\,\abs{Y_{\alpha,2}(s)}_H^{2p}
        =c_p \left(\sum_{k=1}^\infty \int_0^s \sigma^{-2\alpha} \abs{\bar{S}^\vee(\sigma)Qe_k }_H^2 d\sigma \right)^{p} \leq c_{p,T},\ \ \ \ \ \ s \in\,[0,T].
    \end{equation*}
    Hence, for any $\eta > 0$, there exists $\tau_{ \eta}>0$ such that
    \begin{equation*}
        c\,\sqrt{\tau_\eta}\left (\int_0^T \mathbb{E} \abs{Y_{\alpha,2}(s)}_H^{2p} ds \right)^{1/2} \leq  \eta.
    \end{equation*}
 As a consequence of Theorem \ref{prevconv} and the Dominated Convergence theorem, we have 
    \begin{equation*}
        \lim_{\epsilon \to 0} \int_0^T \mathbb{E} \sup_{r \in [\tau_\eta,T]} \abs{[S_\epsilon(r)-\bar{S}(r)^\vee]Y_{\alpha,2}(s)}_H^p ds = 0.
    \end{equation*}
    Since $\eta>0$ is arbitrary, we conclude that
    \begin{equation*}
        \lim_{\epsilon \to 0}\mathbb{E} \sup_{t \in\,[0,T]}\int_0^t \abs{\left[S_\epsilon(t-s)-\bar{S}(t-s)^\vee\right]Y_{\alpha,2}(s)}_H^p ds = 0,
    \end{equation*}
    and this completes the proof of \eqref{star-2}.

\section{The large deviation principle}
\label{secLDP}

In what follows, we will consider the following equation in the domain $G$ \begin{equation}\label{Mai
nSPDE}
    \begin{dcases}
        \displaystyle \partial_t u_\epsilon  (t,x,y)
        = \mathcal{L}_\epsilon u_\epsilon (t,x,y)
        + b(u_\epsilon(t,x,y))
        + \epsilon^{\,\delta/2}\,\partial_t w^Q (t,x,y),
\\[10pt]
        \displaystyle \frac{\partial u_\epsilon }{\partial \nu_\epsilon} (t,x,y) = 0 , (x,y) \in 	\partial G, \ \ \ \ \ \ \ 
        u_\epsilon(0,x,y)=u_0 (x,y).
    \end{dcases}
\end{equation}
As we explained in the Introduction, our goal is proving that   the family $\{\mathcal{L}( u^\epsilon  )\}_{\e \in\,(0,1)}$ satisfies a large deviation principle in the space $C([\tau,T];H)$, or, equivalently, $\{\mathcal{L}( (u^\epsilon )^\wedge )\}_{\e \in\,(0,1)}$ satisfies a large deviation principle in the space $C([\tau,T];\bar{H})$, for every $\tau \in\,(0,T)$. 

To this purpose, we first need to  introduce some notations. For every $M>0$, we denote 
\[    \mathcal{S}_M = \left\{  \varphi \in L^2(0,T;H)\ :\ \abs{\varphi}_{L^2(0,T;H)} \leq M \right\}.\]
Moreover, we denote
\[    \mathcal{P}_M = \left\{  \varphi \in \mathcal{P}\ :\ \varphi \in \mathcal{S}_M, \ \mathbb{P}-a.s. \right\},\]
where  $\mathcal{P}$ is the subspace of predictable processes in $L^2(\Omega; L^2(0,T;H))$.

Next, for every $\varphi \in L^2(0,T;H)$, we consider the following controlled equations in $G$
\begin{equation}\label{control narrow SPDE}
\begin{aligned}
    \begin{dcases}
        \displaystyle \partial_t u_\epsilon^\varphi (t,x,y)
        = \mathcal{L}_\epsilon u_\epsilon^\varphi (t,x,y)
        + b(u_\epsilon^\varphi(t,x,y))
        + (Q\varphi) (t,x,y)
        + \epsilon^{\,\delta/2}\,\partial w^Q_t,\\[10pt]
        \displaystyle \frac{\partial u_\epsilon^\varphi }{\partial \nu_\epsilon} (t,x,y) = 0,\ \ \  (x,y) \in 	\partial G, \ \ \ \ \ \ \ 
        u_\epsilon^\varphi(0,x,y)=u_0 (x,y).
    \end{dcases}
\end{aligned}
\end{equation} 
The same arguments used to prove the well-posedness of  equation \eqref{Mai nSPDE} can be used to prove that for every  $u_0 \in H$ and $\varphi \in L^2(0,T;H)$, there exists a unique mild solution $u_{\epsilon}^{\varphi}(t)$ to \eqref{control narrow SPDE} in $L^p(\Omega,C([0,T];H))$, for $p \geq 1$.
\begin{Lemma}\label{Lp estimate}
    Under Hypotheses \ref{hypB} and \ref{finalhypQ}, there exists some $c_{T, p}>0$ such that    \begin{equation}\label{star-12}
     \sup_{\epsilon \in (0,1)} \mathbb{E} \sup_{t \in [0,T]} \abs{u_{\epsilon}^{\varphi}(t)}^p_{H} \leq c_{T,p} 
    \left(1+\abs{u_0}_H^p+\abs{\varphi}_{L^2(0,T;H)}^{p} \right).    
\end{equation}
\end{Lemma}
\begin{proof}
Since $S_\epsilon(t)$ is contraction on $H$ and $B:H\to H$ is Lipschitz continuous, we have
\begin{equation*}
 \int_0^t\abs{B(u_{\epsilon}^{\varphi}(s))}^p_{H} ds
    \leq c_{T} \left (1+ \int_0^t\abs{u_{\epsilon}^{\varphi}(s)}^p_{H}\,  ds \right),\ \ \ \ \ \ t \in\,[0,T].
\end{equation*}
Therefore, since    
\begin{equation*}
    u_{\epsilon}^{\varphi}(t)
    =\int_0^t S_\epsilon(t-s) B(u_{\epsilon}^{\varphi}(s))\,ds+R_\e(t),
\end{equation*}
where
\[R_\e(t):=S_\epsilon(t)u_0
    +\int_0^t S_\epsilon(t-s) Q\varphi(s)ds
    +\epsilon^{\,\delta/2} \int_0^t S_\epsilon(t-s) dw^Q(s),\]
from  the Gronwall Lemma, we get 
\begin{equation*}
    \abs{u_{\epsilon}^{\varphi}(t)}^p_{H}
    \leq c_{T, p} 
    \left (1+\sup_{t \in [0,T]} \vert R_\epsilon(t) \vert_H^p\right).
\end{equation*}
In particular, this gives
\begin{equation*}
    \sup_{\epsilon \in (0,1)}  \mathbb{E} \sup_{t \in [0,T]} \abs{u_{\epsilon}^{\varphi}(t)}^p_{H}
    \leq c_{T, p} 
    \left (1+\sup_{\epsilon \in (0,1)} \mathbb{E} \sup_{t \in [0,T]} \vert R_\epsilon(t)\vert_H^p \right).
\end{equation*}
We have
\begin{equation*}
    \sup_{t \in [0,T]} \left|\int_0^t S_\epsilon(t-s) Q\varphi(s)ds\right|_H^p \leq c_{T,p} \left (\int_0^T\abs{\varphi(s)}^2_{H} ds \right )^{p/2} \leq c_{T,p}\,\abs{\varphi}_{L^2(0,T;H)}^{p}.
\end{equation*}
Moreover, by the Burkholder-Davis-Gundy inequality and \eqref{bddeps} (with $\alpha=0$), for every $\epsilon \in\,(0,1)$ we have
\begin{equation*}
     \mathbb{E} \sup_{t \in [0,T]}\left(\epsilon^{\,\delta/2}\,\left|w_\epsilon(t)\right|_H\right)^p 
    \leq   c_p   \left(   \int_0^t \Vert S_\epsilon(t-s)\Vert^2_{\mathcal{L}_2(U_0,H)} ds \right)^{p/2}\leq c_{T,p}.
\end{equation*}
Thus, since $\vert S_\e(t)u_0\vert_H^p\leq \vert u_0\vert_H^p$,
we conclude that
\begin{equation*}
    \sup_{\epsilon \in (0,1)} \mathbb{E} \sup_{t \in [0,T]} \abs{u_{\epsilon}^{\varphi}(t)}^p_{H}
    \leq c_{T,p} 
    \left(1+\abs{u_0}_H^p+\abs{\varphi}_{L^2(0,T;H)}^{p}\right).
\end{equation*}
\end{proof}

Next, we consider  the following controlled equation on the graph $\Gamma$
\begin{equation}\label{control graph SPDE}
    \begin{dcases}
        \displaystyle \partial_t\bar{u}^{\varphi}(t,x,k) = \bar{L}\bar{u}^{\varphi}(t,x,k)+b(\bar{u}^{\varphi}(t,x,k))+ (Q\varphi)^\wedge (t,x,k),\\[10pt]
        \displaystyle \bar{u}^{\varphi}(0,x,k)=u_0^\wedge(x,k).
    \end{dcases}
\end{equation}
As for equation \eqref{graph SPDE}, we have that for every $u_0 \in H$, $\varphi \in L^2(0,T;H)$ and $p\geq 1$, there exists a unique solution $\bar{u}^{\varphi}(t) \in\,L^p(\Omega,C([0,T];\bar{H}))$ to  \eqref{control graph SPDE}. Moreover,  
    \begin{equation}\label{star-13}
        \mathbb{E} \sup_{t \in [0,T]} \abs{\bar{u}^{\varphi}(t)}^p_{\bar{H}} \leq  c_{T,p} \left(1+\abs{u_0}_H^p+\abs{\varphi}_{L^2(0,T;H)}^{p} \right).
    \end{equation}

Now we can state our main theorem.
\begin{Theorem}\label{LDPT}
    Assume Hypotheses \ref{hypB}, and \ref{finalhypQ}. Then, for every $ 0 < \tau < T$, the family $\{\mathcal{L}(  u_\epsilon )\}_{\epsilon \in\,(0,1)}$ satisfies a large deviation principle in $C([\tau,T];H)$ with speed $\e^\delta$ and action functional 
    \begin{equation*}
       I_T(f)=\frac{1}{2}\, \inf  \abs{\varphi}_{L^2(0,T;H)}^2,
    \end{equation*}
    where the infimum is taken over all $\varphi \in L^2(0,T;H)$ such that $f=(\bar{u}^\varphi)^\vee|_{[\tau,T]} \in \,C([\tau,T];H)$ and $\bar{u}^\varphi$ is the solution to the controlled problem \eqref{control graph SPDE}.
\end{Theorem}

As shown in \cite{LDP}, Theorem \ref{LDPT} can be obtained by showing the following.

\begin{enumerate}
\item[Claim 1.]  For all $M>0$ and all $\{ \varphi_\epsilon \} \subset \mathcal{P}_M$ and $\varphi \in \mathcal{P}_M$ such that $\varphi_\epsilon \to \varphi$ in distribution with respect to the weak topology in $L^2(0,T;H)$ we have
   \begin{equation*}
       \lim_{\epsilon \to 0} \,u_\epsilon^{\varphi_\epsilon}
       = (\bar{u}^\varphi)^\vee,
   \end{equation*}
   in distribution in $C([\tau,T];H)$, for all $\tau \in\,(0,T)$.
   \item[Claim 2.]	For all $T,R>0$, the level set $\Phi_{T,R}:=\{  I_T \leq R \}$ is compact in $C([\tau,T];H)$, for all $\tau>0$.
\end{enumerate}

\begin{Remark}
{\em Notice that, due to Proposition \ref{opprop}, if Claim 1. holds then   
\begin{equation*}
       \lim_{\epsilon \to 0} \,(u_\epsilon^{\varphi_\epsilon})^\wedge
       = \bar{u}^\varphi,
   \end{equation*}
   in distribution in $C([\tau,T];\bar{H})$, for all $\tau \in\,(0,T)$. Moreover, if we define
   \begin{equation*}\bar{I}_T(f)=\frac{1}{2}\, \inf  \abs{\varphi}_{L^2(0,T;H)}^2,
    \end{equation*}
    where the infimum is taken over all $\varphi \in L^2(0,T;H)$ such that $f=\bar{u}^\varphi|_{[\tau,T]} \in \,C([\tau,T];\bar{H})$, from the compactness of $\Phi_{T, R}$ in $C([\tau,T];H)$ we get the compactness of 
$\bar{\Phi}_{T,R}:=\{\bar{I}_T\leq R\}$ in     $C([\tau,T];\bar{H})$.

In particular, if we prove that Claim 1. and Claim 2. we also get that the family $\{\mathcal{L}(  (u_\epsilon)^\wedge )\}_{\epsilon \in\,(0,1)}$ satisfies a large deviation principle in $C([\tau,T];\bar{H})$ with action functional $\bar{I}_T$.

}	
\end{Remark}

\section{From weak convergence to strong convergence}\label{secpriori}
We fix $M>0$ and consider a sequence $\{ \varphi_\epsilon \}_{\epsilon \in (0,1)} \subset \mathcal{P}_M$ and $\varphi \in \mathcal{P}_M$, such that $\varphi_\epsilon \to \varphi$, as $\epsilon\downarrow 0$, almost surely in $L_w^2(0,T;H)$, where $L_w^2(0,T;H)$ denotes the space $L^2(0,T;H)$, endowed with  the weak topology. 

In this section, our  goal in showing that
\begin{equation}
\label{star-14}
\lim_{\e\to0}\mathbb{E}\,\sup_{t \in\,[0,T]}\vert \Phi_\epsilon(t)-\bar{\Phi}(t)^\vee\vert_H^2=0,	
\end{equation}
where
\begin{equation} \label{star-31}
    \Phi_\epsilon(t) \coloneqq \int_0^t S_\epsilon(t-s)Q\varphi_\epsilon(s) ds,\ \ \ \ \ \  \bar{\Phi}(t):=\int_0^t \bar{S}(t-s)(Q\varphi(s))^\wedge ds .\end{equation}

\begin{Lemma}\label{Energy}
Under Hypotheses \ref{hypB} and \ref{finalhypQ}, for every $\epsilon \in\,(0,1)$ we have
    \begin{equation}\label{star-15}
        \abs{\Phi_\epsilon}_{L^\infty([0,T];H^1)}
        +\abs{\partial_t \Phi_\epsilon}_{L^2(0,T;H)} \leq c_{M,T},\ \ \ \ \ \mathbb{P}-\text{a.s.}
    \end{equation}
\end{Lemma}
\begin{proof}
 The random function $\Phi_\e$ is a weak solution of the  random problem in $G$
 \begin{equation}
    \begin{dcases}
        \displaystyle \partial_t \Phi_\epsilon  (t,x,y)
        =\mathcal{L}_\epsilon \Phi_\epsilon (t,x,y)
        + Q\varphi_\epsilon (t,x,y) \\[10pt]
        \displaystyle \frac{\partial \Phi_\epsilon }{\partial \nu_\epsilon} (t,x,y) = 0,\ \ \  (x,y) \in 	\partial G,\ \ \ \ \ \ \  
        \Phi_\epsilon(0,x,y)=0.
    \end{dcases}
\end{equation}
It is immediate to check that
\begin{equation*}
        \sup_{0 \leq t \leq T}\abs{\Phi_\epsilon(t)}_H^2 
        \leq c_T \int_0^T \abs{S_\epsilon(t-s)Q\varphi_\epsilon(s) ds}_H^2 ds
        \leq c_{T} M^2.
    \end{equation*}
Moreover, if we define $\nabla_\epsilon:=(\partial_x,\epsilon^{-1}\,\partial_y)$, and $E_\e(t):= \frac 12\vert\nabla_\epsilon\Phi_\epsilon(t)\vert_H^2,$ we  have
    \begin{equation*}
    \begin{aligned}
         \frac{d E_\epsilon}{dt}(t) = \langle \nabla_\epsilon& \Phi_\epsilon, \nabla_\epsilon \partial_t \Phi_\epsilon \rangle_H
        = - 2\langle \mathcal{L}_\epsilon \Phi_\epsilon , \partial_t \Phi_\epsilon \rangle_H =  - 2\langle \partial_t \Phi_\epsilon , \partial_t \Phi_\epsilon \rangle_H 
        + 2\langle Q\varphi_\epsilon , \partial_t \Phi_\epsilon \rangle_H \\[10pt]
        &\quad \quad \leq  - 2\,\vert \partial_t \Phi_\epsilon\vert_H^2 
        + \vert Q\varphi_\epsilon\vert_H^2
        + \vert \partial_t \Phi_\epsilon \vert_H^2,
    \end{aligned}
    \end{equation*}
    so that
        \begin{equation*}
        \frac{dE_\epsilon}{dt}(t) + \vert \partial_t \Phi_\epsilon\vert_H^2
         \leq \vert  Q\varphi_\epsilon\vert_H^2
         \leq c\, M^2,\ \ \ \ \ \ \mathbb{P}-\text{a.s.}
    \end{equation*}
By integrating both sides, this gives
    \begin{equation*}
        E_\epsilon(t) -E_\epsilon(0) + \int_0^t \vert \partial_t \Phi_\epsilon\vert_H^2 ds \leq c\, M^2t,
    \end{equation*}
    and hence
    \begin{equation*}
        \sup_{0 \leq t \leq T}  \abs{\nabla_\epsilon \Phi_\epsilon}_H^2 +   
       \abs{\partial_t \Phi_\epsilon}_{L^2(0,T;H)}^2  \leq c\, M^2T.
    \end{equation*}
    Since we $\abs{\nabla \Phi_\epsilon}_H^2\leq \abs{\nabla_\epsilon \Phi_\epsilon}_H^2$, for $\epsilon \in (0,1)$,
    this implies  \eqref{star-15}.
\end{proof}
As a consequence of Lemma \ref{Energy}, we have 
\begin{equation}
\label{star-16}	
\{\Phi_\epsilon\}_{\e \in\,(0,1)}\subset \Lambda_{T,M},\ \ \ \ \ \mathbb{P}-\text{a.s.}
\end{equation}
where $\Lambda_{T,M}$ is the set of all $\Phi \in\,L^\infty(0,T;H^1)$, such that there exists $\partial_t\Phi \in\,L^2(0,T;H)$, with
\[\sup_{t \in\,[0,T]}\vert\Phi(t)\vert_{H^1}^2+\int_0^T\vert\partial_t\Phi(s)\vert_H^2\,ds\leq c_{T,M},\]
for some $c_{T, M}>0$.
Notice that by the Aubin-Lions Lemma, the set
$\Lambda_{T, M}$ is compactly embedded in  $C([0,T];H)$.

\begin{Lemma}\label{wtsg}
 Under Hypotheses \ref{hypB}, and \ref{finalhypQ}, we have
    \begin{equation}\label{convsemi} 
  \lim_{\epsilon\to 0}\,  \sup_{t \in [0,T]}\abs{\Phi_\e(t)
    -\Phi^\e(t)}_H = 0,\ \ \ \ \ \mathbb{P}-\text{a.s.} 
    \end{equation}
    where
    \[\Phi^\epsilon(t):=\int_0^t  \bar{S}(t-s)^\vee Q\varphi_\epsilon(s) \,ds.\]   
\end{Lemma}
\begin{proof}
We have
\[\begin{aligned}
&\vert \Phi_\e(t)
    -\Phi^\e(t)\vert_H\leq \int_0^t\vert (S_\e(t-s)-\bar{S}(t-s)^\vee)Q\varphi_\e(s)\vert_H\,ds	\\[10pt]
    &\leq \left(\int_0^t \Vert (S_\e(s)-\bar{S}(s)^\vee)Q\Vert_H^2\,ds\right)^{1/2}\!\!\vert \varphi_\epsilon\vert_{L^2(0,T;H)}\leq \left(\int_0^t \Vert S_\e(s)-\bar{S}(s)^\vee\Vert_{\mathcal{L}_2(U_0,H)}^2\,ds\right)^{1/2}\!\!M.
\end{aligned}
\]
Thus, we obtain \eqref{convsemi} due to \eqref{star-20}.
\end{proof}

Now we can prove the following result.
\begin{Lemma}
    For every $\psi \in L^2(0,T;H)$, we have
    \begin{equation*}
    \begin{aligned}
       \lim_{\e\to 0} \langle \Phi_\epsilon-\bar{\Phi}^\vee ,\psi\rangle_{L^2(0,T;H)} =0.    \end{aligned}
    \end{equation*}
\end{Lemma}
\begin{proof} Under Hypotheses \ref{hypB} and \ref{finalhypQ}, we have
    \begin{equation*}
    \begin{aligned}
    \langle \Phi_\epsilon-&\bar{\Phi}^\vee ,\psi\rangle_{L^2(0,T;H)}=\langle \Phi_\epsilon-\Phi^\e,\psi\rangle_{L^2(0,T;H)}+\langle \Phi^\epsilon-\bar{\Phi}^\vee ,\psi\rangle_{L^2(0,T;H)}=:I_{\e,1}+I_{\e,2}.\end{aligned}
    \end{equation*}
      By Lemma \ref{wtsg}, we have that
      \[\lim_{\e\to 0} I_{\epsilon,1}=0,\ \ \ \ \ \mathbb{P}-\text{a.s.}\]     
 Hence, we only need to prove that
 \begin{equation}\label{star-22}\lim_{\e\to 0} I_{\epsilon, 2}=0,\ \ \ \ \ \mathbb{P}-\text{a.s.}\end{equation}
If for every $\psi \in\,L^2(0,T;H)$, we define
\[R_\psi(\varphi) \coloneqq \int_0^T\langle\int_0^t \bar{S}(t-s)^\vee Q\varphi(s)\,ds, \psi(t) \rangle_{H}\,dt,\] 
we have that
\[I_{\epsilon,2}=R_\psi(\varphi_\e-\varphi).\]
It is immediate to check  that the linear operator $R_\psi$ is  bounded  in $L^2(0,T;H)$. Then,  since $\varphi_\e$ converges $\mathbb{P}$-a.s. to $\varphi$  in $L_w^2(0,T;H)$, we conclude that \eqref{star-22} holds. 
\end{proof}

We have just proved that $\Phi_\epsilon$ converges weakly to $\bar{\Phi}^\vee$ in $L^2(0,T;H)$. Moreover, as we have seen above
\[\{\Phi_\e\}_{\epsilon \in\,(0,1)}\subset \Lambda_{T,M},\ \ \ \ \ \mathbb{P}-\text{a.s.}\] 
Due to the compactness of $\Lambda_{T,M}$ in $C([0,T];H)$, for every sequence $  \epsilon_n \to 0$, there exists a subsequence $ \{ \epsilon_{n_k} \} \subset \{ \epsilon_n \}$ and $\Tilde{\Phi} \in C([0,T];H)$ such that $\Phi_{\epsilon_{n_k}} \to \Tilde{\Phi}$ in $C([0,T];H)$.
By the uniqueness of the limit, $\Tilde{\Phi} = \bar{\Phi}^\vee$ and thus we obtain  the following result.

\begin{Proposition} Assume Hypotheses \ref{hypB} and \ref{finalhypQ}. Then, for every sequence $\{ \varphi_\epsilon \}_{\epsilon \in (0,1)} \subset \mathcal{P}_M$ and $\varphi \in \mathcal{P}_M$, such that $\varphi_\epsilon$ converges to $\varphi$, as $\epsilon\downarrow 0$, almost surely,  in $L_w^2(0,T;H)$, we have
\begin{equation*}
  \lim_{\e\to 0}\sup_{t \in\,[0,T]}\left|  \int_0^t S_\epsilon(t-s) Q\varphi_\epsilon(s) ds-  \int_0^t \bar{S}^\vee(t-s) Q\varphi(s) ds\right|_H=0,\ \ \ \ \ \ \mathbb{P}-\text{a.s.}
\end{equation*}
\end{Proposition}
In particular, 
since $ \{\varphi_\epsilon \}, \varphi \subset \mathcal{P}_M$ and $\varphi \in\,\mathcal{P}_M$, by the  Dominated Convergence theorem we have
    \begin{equation}\label{wts}
        \lim_{\epsilon \to 0} \mathbb{E} \sup_{t \in [0,T]} \abs{\int_0^t S_\epsilon(t-s) Q\varphi_\epsilon(s) ds -   \int_0^t \bar{S}^\vee(t-s) Q\varphi(s) ds}_H^2 =0.   
    \end{equation}

\section{Proof of Theorem \ref{LDPT}} 
\label{secpf}

As we have seen in Section \ref{secLDP}, the proof of Theorem \ref{LDPT} follows from the proof of Claim 1. and Claim 2.

\subsection{Proof of Claim 1.}
\label{secpf1}

Let  $\{\varphi_\epsilon\}_{\epsilon>0}$ be an arbitrary family of processes in $ \mathcal{P}_{M}$ converging in distribution, with respect to the weak topology of $L^2(0,T;H)$, to some 
  $\varphi \in\,\mathcal{P}_{M}$.  As a consequence of the  Skorohod theorem, we can assume that the sequence $\{\varphi_\epsilon\}_{\epsilon>0}$  converges $\mathbb{P}$-a.s. to $\varphi$, with respect to the weak topology of  $L^2(0,T;H)$. We will prove that this implies that 
\begin{equation}\label{star-28}
    \lim_{\epsilon \to 0} \mathbb{E} \sup_{t\in [\tau,T]}\abs{ u_{\epsilon}^{\varphi_\epsilon}(t)-(\bar{u}^\varphi)^\vee(t) }_{H}^2 = \lim_{\epsilon \to 0} \mathbb{E} \sup_{t\in [\tau,T]}\abs{ (u^{\varphi_\e}_{\epsilon})^\wedge(t)-\bar{u}^\varphi(t) }_{\bar{H}}^2 =  0.
\end{equation} 
Due to \eqref{star-7} and \eqref{star-30}, we have
\[(\bar{S}(t-s)B(\bar{u}^\varphi(s)))^\vee=(\bar{S}(t-s)(B(\bar{u}^\varphi(s))^\vee)^\wedge)^\vee=\bar{S}^\vee(t-s)B((\bar{u}^\varphi)^\vee(s)).\]
Thus, with the notations introduced in \eqref{star-30} and \eqref{star-31}, we have
\begin{equation*}
\begin{aligned}
  \vert u_{\epsilon}^{\varphi_\epsilon}(t)-&(\bar{u}^\varphi)^\vee(t) \vert_{H}^2
    \leq c_T \Big(  \int_0^t \abs{  S_\epsilon(t-s) B(u_{\epsilon}^{\varphi_\epsilon}(s))
    - S_\epsilon(t-s) B((\bar{u}^\varphi)^\vee(s))  }_{H}^2 ds\\[10pt]
    & +\abs{ S_{\epsilon}(t)u_0-\bar{S}^\vee(t)u_0 }_{H}^2  +\int_0^t \abs{  S_\epsilon(t-s) B((\bar{u}^\varphi)^\vee(s))
    - \bar{S}^\vee(t-s) B((\bar{u}^\varphi)^\vee(s))  }_{H}^2 ds\\[10pt]
    &\quad \quad \quad \quad \quad  +\abs{\Phi_\epsilon(t)-\bar{\Phi}^\vee(t)}_{H}^2  + \e^\delta\,\abs{  w_\epsilon(t) }_{H}^2 \Big)  =:c_T\,\sum_{k=1}^5 I_{k,\epsilon}(t).
\end{aligned}
\end{equation*} 
For $I_{1,\epsilon}(t)$, since $S_\epsilon(t)$ is a contraction  on $H$ and $B$ is Lipschitz continuous, we have
\begin{equation*}
\begin{aligned}
    I_{1,\epsilon}(t) \leq  \int_0^t \abs{  B(u_{\epsilon}^{\varphi_\epsilon}(s))
    -  B((\bar{u}^\varphi)^\vee(s))  }_{H}^2 ds \leq c \int_0^t \abs{  u_{\epsilon}^{\varphi_\epsilon}(s)
    -  (\bar{u}^\varphi)^\vee(s)  }_{H}^2 ds.
\end{aligned}
\end{equation*}
Thanks to Gr\"{o}nwall's inequality, this implies
\begin{equation}\label{star-27}
    \abs{ u_{\epsilon}^{\varphi_\epsilon}(t)-(\bar{u}^\varphi)^\vee(t) }_{H}^2 \leq c_T \left(R_\epsilon(t)+\int_0^t R_\epsilon(s)ds \right),
\end{equation}
where
\[R_\e(t):=c_T\,\sum_{k=2}^5 I_{k,\epsilon}(t).\]
In particular, for every $\tau \in\,(0,T)$, we get
\begin{equation}\label{Gronwall re}
    \mathbb{E} \sup_{t \in [\tau,T]} \abs{ u_{\epsilon}^{\varphi_\epsilon}(t)-(\bar{u}^\varphi)^\vee(t) }_{H}^2 \leq c_T \left( \mathbb{E} \sup_{t \in [\tau,T]} R_\epsilon(t)+ \int_0^T \mathbb{E} \,R_\epsilon(s) ds \right).
\end{equation}

In view of  Theorem \ref{prevconv}, we have
\begin{equation}\label{MTconv1}
    \lim_{\e\to 0} \mathbb{E} \sup_{t \in [\tau,T]} I_{2,\epsilon}(t) 
    =\lim_{\e\to 0} \sup_{t \in [\tau,T]}\abs{ S_{\epsilon}(t)u_0-\bar{S}^\vee(t)u_0 }_{H}^2=0.
\end{equation}
Next, if we fix an arbitrary $\vartheta \in\,(0,T)$, we have
\begin{equation}
\label{star-25}
\mathbb{E}\sup_{t \in\,[0,\vartheta]}\,I_{3,\epsilon}(t)\leq 	c\,\int_0^\vartheta \left( 1+\abs{ \bar{u}^\varphi(s)}_{\bar{H}}^2 \right) ds\leq \vartheta\, c\,  \left(1+ \abs{ \bar{u}^\varphi}_{C([0,T];\bar{H})}^2 \right).
\end{equation}
Moreover, if $t \in\,(\vartheta,T]$, we have
\begin{equation*}
\begin{aligned}
    I_{3,\epsilon}(t) 
    =& \int_0^{t-\vartheta} \abs{  S_\epsilon(t-s) B((\bar{u}^\varphi)^\vee(s))
     - \bar{S}^\vee(t-s) B((\bar{u}^\varphi)^\vee(s))  }_{H}^2 ds \\[10pt]
     & \quad \quad \quad \quad \quad +  \int_{t-\vartheta}^t \abs{  S_\epsilon(t-s) B((\bar{u}^\varphi)^\vee(s))
     - \bar{S}^\vee(t-s) B((\bar{u}^\varphi)^\vee(s))  }_{H}^2 ds \\[10pt]
    \leq &\int_0^T \sup_{r \in [\vartheta,T]} \abs{  S_\epsilon(r) B((\bar{u}^\varphi)^\vee(s))
     - \bar{S}^\vee(r) B((\bar{u}^\varphi)^\vee(s))  }_{H}^2 ds +c \int_{t-\vartheta}^t \left( 1+\abs{ \bar{u}^\varphi(s)}_{\bar{H}}^2 \right) ds \\[10pt]
     \leq & \int_0^T \sup_{r \in [\vartheta,T]} \abs{  S_\epsilon(r) B((\bar{u}^\varphi)^\vee(s))
     - \bar{S}^\vee(r) B((\bar{u}^\varphi)^\vee(s))  }_{H}^2 ds +\vartheta\, c\,  \left(1+ \abs{ \bar{u}^\varphi}_{C([0,T];\bar{H})}^2 \right).  
\end{aligned}
\end{equation*}
This, together with \eqref{star-25}, implies that
\[
\begin{aligned}
\mathbb{E}&\sup_{t \in\,[0,T]}\,I_{3,\epsilon}(t)\\[10pt]
&\leq
\mathbb{E}\int_0^T \sup_{r \in [\vartheta,T]} \abs{  S_\epsilon(r) B((\bar{u}^\varphi)^\vee(s))
     - \bar{S}^\vee(r) B((\bar{u}^\varphi)^\vee(s))  }_{H}^2 ds +\vartheta\, c\,  \left(1+ \mathbb{E}\,\abs{ \bar{u}^\varphi}_{C([0,T];\bar{H})}^2 \right).	
\end{aligned}
\]
Thanks to \eqref{star-13}, for $\eta > 0$, there exists $\vartheta_\eta>0$ such that
\begin{equation*}
    \vartheta_\eta\, c\,  \left(1+ \mathbb{E}\,\abs{ \bar{u}^\varphi}_{C([0,T];\bar{H})}^2 \right) < \frac{\eta}{2}.
\end{equation*}
Then, by Theorem \ref{prevconv}, we can pick $\epsilon_\eta > 0$, such that 
\begin{equation*}
    \mathbb{E} \int_0^T \sup_{r \in [\vartheta_\eta,T]} \abs{  S_{\epsilon}(r) B((\bar{u}^\varphi)^\vee(s))
     - \bar{S}^\vee(r) B((\bar{u}^\varphi)^\vee(s))  }_{H}^2 ds < \frac{\eta}{2},\ \ \ \ \ \e\leq \e
     _\eta,
\end{equation*}
and due to the arbitrariness of $\eta>0$, this allows to conclude that 
\begin{equation}\label{MTconv3}
   \lim_{\epsilon\to 0} \mathbb{E} \sup_{t \in [0, T]}I_{3,\epsilon}(t) =  0.
\end{equation}
As a consequence of  \eqref{wts}, we have 
\begin{equation}\label{MTconv4}
    \lim_{\epsilon\to 0}\,\mathbb{E} \sup_{t \in [0,T]} I_{4,\epsilon}(t)=0.
\end{equation}
Moreover, due to  \eqref{bddeps}, we have
\[\mathbb{E} \sup_{t \in [0,T]} I_{5,\epsilon}(t) 
    = \epsilon^{\,\delta}\,\mathbb{E} \sup_{t \in [0, T]} \abs{  \int_0^t S_\epsilon(t-s) dw^Q(s)  }_{H}^2 
    \leq c\, \epsilon^{\,\delta} \int_0^T \Vert S_\epsilon(s)\Vert_{\mathcal{L}_2(U_0,H)}^2 ds,\]
    so that
\begin{equation}\label{MTconv5}
\lim_{\e\to 0} \mathbb{E} \sup_{t \in [0,T]} I_{5,\epsilon}(t) =0.     
\end{equation}

Therefore, if we combine together \eqref{MTconv1}, \eqref{MTconv3}, \eqref{MTconv4}, \eqref{MTconv5}, we obtain
\begin{equation}\label{re conv 0}
    \lim_{\epsilon \to 0} \mathbb{E} \sup_{t \in [\tau,T]} R_\epsilon(t) =0,\ \ \ \ \ \tau \in\,(0,T).
\end{equation}
In particular, since 
\[\sup_{\epsilon \in\,(0,1)}\mathbb{E}\sup_{t \in [\tau,T]} R_\epsilon(t) <\infty,\]
by the Dominated Convergence theorem we get
\[\lim_{\epsilon\to 0}\int_0^T\mathbb{E} R_\epsilon(t)\,dt=0.\]
Due to \eqref{star-27}, this together with \eqref{re conv 0} implies \eqref{star-28}.

\subsection{Proof of Claim 2.}

In the proof of Claim 1. we have seen that if $\varphi_\e$ converges $\mathbb{P}$-a.s. to $\varphi$, in $L^2_w(0,T;H)$, where we recall $L^2_w(0,T;H)$ denotes the space $L^2(0,T;H)$ endowed with  the weak topology, then \eqref{star-28} holds. In particular, \eqref{star-28} holds in the deterministic case, so that 
 the mapping
\[\varphi \in\,L^2_w(0,T;H)\mapsto (\bar{u}^\varphi)^\vee \in\,C([0,T];H),\]
is continuous,  and for every $c>0$
\begin{equation}
\label{r3-4}
\bigcap_{\lambda \in\,(0,1)}\left\{(\bar{u}^\varphi)^\vee\, ,\, \varphi \in\,\mathcal{S}_{c+\lambda}\right\}=	\left\{(\bar{u}^\varphi)^\vee\, ,\, \varphi \in\,\mathcal{S}_{c}\right\}.
\end{equation}
 Moreover,  the set $ {\mathcal{S}}_{c}$ is compact in $L^2_w(0,T;H)$, so that
\[  {\left\{ (\bar{u}^\varphi)^\vee \, ,\, \varphi \in\, {\mathcal{S}}_{c}\right\}}\subset C([0,T];H)\ \ \text{is compact}.\] 
Now, recalling the definition of $I_{T}$,  for every $R>0$ we have
\begin{equation}\label{r3-5}\Phi_{T,R}=\{I_{T}\leq R\}=\{ (\bar{u}^\varphi)^\vee \, ,\, \varphi \in\, {\mathcal{S}}_{{\sqrt{2R}}}\}.\end{equation}
 {Indeed, if $u$ belongs to $\{\bar{u}^\varphi \ :\ \varphi \in\,\mathcal{S}_{\sqrt{2R}}\}$, then there exists $\hat{\varphi} \in\, {\mathcal{S}}_{{\sqrt{2R}}}$ such that $u=(\bar{u}^{\hat{\varphi}})^\vee$, so that $I_{T}(u)\leq R$. On the other hand, if $u \in\,\Phi_{T,R}$, then for any $\lambda>0$ there exists $\varphi_\lambda \in\, {\mathcal{S}}_{{\sqrt{2R}+\lambda}}$ such that $u=(\bar{u}^{\varphi_\lambda})^\vee$, and together with  \eqref{r3-4} this implies 
\[u \in\,\bigcap_{\lambda \in\,(0,1)}\left\{(\bar{u}^\varphi)^\vee\, ,\, \varphi \in\,\mathcal{S}_{\sqrt{2R}+\lambda}\right\}=	\left\{(\bar{u}^\varphi)^\vee\, ,\, \varphi \in\,\mathcal{S}_{\sqrt{2R}}\right\}.\]
Therefore,  from \eqref{r3-5}  and the compactness of $\{(\bar{u}^\varphi)^\vee \, ,\, \varphi \in\, {\mathcal{S}}_{{\sqrt{2R}}}\}$}, we conclude that Claim 2. holds.

\end{document}